\pdfoutput=1
\documentclass[11pt,a4paper]{article}
\usepackage[utf8]{inputenc} 
\usepackage[english]{babel} 

\usepackage{lmodern}
\usepackage[T1]{fontenc}
\usepackage{amsmath,mathtools,amsfonts,amssymb} 
\usepackage{slashed} 
\usepackage{mathrsfs}
\usepackage[babel=true]{microtype}
\usepackage[autostyle=true]{csquotes} 

\usepackage{comment}

\usepackage[a4paper,marginparwidth=2cm]{geometry}

\usepackage{amsthm}
\usepackage{thmtools}
\usepackage[pdfusetitle]{hyperref}
\usepackage[citestyle=numeric-comp,
            bibstyle=numeric-comp,
            backend=biber,
            url=false,
            doi=true,
            isbn=false,
            giveninits=true,
            maxbibnames=100,
            sorting=nyt]{biblatex}
\AtEveryBibitem{\clearfield{month}}
\AtEveryBibitem{\clearfield{day}}
\AtEveryBibitem{%
  \ifentrytype{thesis}
    {}
    {
      \ifentrytype{online}{}
      {
      \clearfield{url}%
      \clearfield{urldate}%
      }
    }%
}
            
\usepackage{tikz}
\usetikzlibrary{cd,calc,patterns,intersections} 

\usepackage[shortlabels]{enumitem}
\newlist{myenumi}{enumerate}{1}
\setlist[myenumi,1]{label=\upshape(\roman*)}
\newlist{myenuma}{enumerate}{1}
\setlist[myenuma,1]{label=\upshape(\alph*)}
\usepackage{color}
\usepackage{todonotes}
\declaretheorem[name=Theorem, numberwithin=section]{thm}
\declaretheorem[name=Theorem, numbered=no]{thm*}
\declaretheorem[name=Lemma,numberlike=thm]{lem}
\declaretheorem[name=Lemma,numbered=no]{lem*}
\declaretheorem[name=Corollary,numberlike=thm]{cor}
\declaretheorem[name=Proposition,numberlike=thm]{prop}
\declaretheorem[name=Definition,numberlike=thm, style=definition]{defi}
\declaretheorem[name=Setup,numberlike=thm, style=definition]{setup}
\declaretheorem[name=Conjecture,numberlike=thm, style=remark]{conj}
\declaretheorem[name=Example, numberlike=thm, style=remark]{ex}
\declaretheorem[name=Remark, numberlike=thm, style=remark]{rem}

\numberwithin{equation}{section}
\allowdisplaybreaks[1]
\usepackage[noabbrev]{cleveref} 
\crefname{figure}{Figure}{Figures}
\crefname{table}{Table}{Tables}
\crefname{thm}{Theorem}{Theorems}
\crefname{lem}{Lemma}{Lemmas}
\crefname{defi}{Definition}{Definitions}
\crefname{setup}{Setup}{Setups}
\crefname{conj}{Conjecture}{Conjectures}
\crefname{cor}{Corollary}{Corollaries}
\crefname{prop}{Proposition}{Propositions}
\crefname{ex}{Example}{Examples}
\crefname{rem}{Remark}{Remarks}
\crefname{section}{Section}{Sections}
\crefname{subsection}{Subsection}{Subsections}
\crefname{chapter}{Chapter}{Chapters}
\crefname{appendix}{Appendix}{Appendices}
\crefdefaultlabelformat{#2\textup{#1}#3}
\creflabelformat{enumi}{(#2#1#3)}
\usepackage{xparse}
\usepackage{ourmacros}

\usepackage{authblk}

\title{Scalar and mean curvature comparison\\via the Dirac operator}
\author{Simone Cecchini\thanks{Funded by the Deutsche Forschungsgemeinschaft (DFG, German Research Foundation) through the Priority Programme ``Geometry at Infinity'' (SPP 2026, CE~393/1-1).}}
\affil{Mathematical Institute\\University of Göttingen, Germany\\\vspace{0.2cm}
email:~\href{mailto:simone.cecchini@mathematik.uni-goettingen.de}{simone.cecchini@mathematik.uni-goettingen.de}\\ url:~\href{https://simonececchini.org}{simonececchini.org}}
\author{Rudolf Zeidler\thanks{Funded by the Deutsche Forschungsgemeinschaft (DFG, German Research Foundation) – Project-ID 427320536 – SFB 1442, as well as under Germany’s Excellence Strategy EXC 2044  390685587, Mathematics Münster: Dynamics–Geometry–Structure, and through the Priority Programme ``Geometry at Infinity'' (SPP 2026, ZE~1123/2-2).}}
\affil{Mathematical Institute\\University of Münster, Germany\\\vspace{0.2cm}
email:~\href{mailto:math@rzeidler.eu}{math@rzeidler.eu}\\ url:~\href{https://www.rzeidler.eu}{www.rzeidler.eu}}
\date{}
\hypersetup{
  colorlinks=true,
  allcolors=blue,
}
\begin{document}

\maketitle
\begin{abstract} 
We use the Dirac operator technique to establish sharp distance estimates for compact spin manifolds under lower bounds on the scalar curvature in the interior and on the mean curvature of the boundary.
In the situations we consider, we thereby give refined answers to questions on metric inequalities recently proposed by Gromov.
This includes optimal estimates for Riemannian bands and for the long neck problem.
In the case of bands over manifolds of non-vanishing $\widehat{\mathrm{A}}$-genus, we establish a rigidity result stating that any band attaining the predicted upper bound is isometric to a particular warped product over some spin manifold admitting a parallel spinor.
Furthermore, we establish scalar- and mean curvature extremality results for certain log-concave warped products.
The latter includes annuli in all simply-connected space forms.
On a technical level, our proofs are based on new spectral estimates for the Dirac operator augmented by a Lipschitz potential together with local boundary conditions.
\end{abstract}
\newpage
\tableofcontents

\section{Introduction}
Manifolds of positive scalar curvature have been a central topic in differential geometry and topology in recent decades.
On complete spin manifolds, a particularly powerful tool in the study of positive scalar curvature metrics has been the spinor Dirac operator which facilitates a fruitful exchange between geometry and topology.
This technique exploits the tension between, on the one hand, the Schrödinger--Lichnerowicz formula
\[
  \ReducedSpinDirac^2 = \nabla^\ast \nabla + \frac{\scal}{4},  
\]
which implies invertibility of the spinor Dirac operator \(\ReducedSpinDirac\) in case the scalar curvature is uniformly positive and, on the other hand, index theory in the sense of Atiyah and Singer which in various situations yields differential-topological obstructions to invertibility.
Until recently, and with the notable exceptions of sharp Dirac eigenvalue estimates~\cite{Friedrich:DiracEigenwert,HOR02}, sharp \(\K\)-area estimates~\cite{LLarull,Goette-Semmelmann}, and approaches to the positive mass theorem based on an idea of Witten~\cite{Witten:positive_mass}, the strongest applications of the Dirac operator technique in positive scalar curvature geometry have been of a fundamentally \emph{qualitative} nature.
Indeed, there is a substantial body of celebrated literature addressing existence questions of positive scalar curvature metrics on a given manifold, or more generally, studying the topology of the space of positive scalar curvature metrics via the Dirac method; see~\cite{Lichnerwociz:Spineurs,GromovLawson:PSCDiracComplete,GromovLawson:Classification,Stolz:SimplyConnected,Botvinnik-Ebert-RW:InfiniteLoopSpaces} for a selection.
However, Gromov~\cite{Gromov:MetricInequalitiesScalar,gromovFourLecturesScalar2019} recently directed the focus towards more \emph{quantitative} questions and proposed studying the geometry of scalar curvature via various  metric inequalities which have similarities to classical Riemannian comparison geometry.
This resulted in a number of conjectures, a few of which we now recall.
\begin{conj}[{\cite[p.~87, \emph{Long neck problem}]{gromovFourLecturesScalar2019v3}}]\label{conj:long_neck}
Let $(M,g)$ be a compact connected \(n\)-dimensional Riemannian manifold with boundary such that its scalar curvature is bounded below by \(n(n-1)\).
Suppose that $\Phi \colon M\to S^n$ is a smooth area non-increasing map which is locally constant near the boundary.
If
\[
  \dist_g(\supp(\dd \Phi), \partial M) \geq \frac{\pi}{n},
\]
then the mapping degree of \(\Phi\) is zero.
\end{conj}

\begin{conj}[{\cite[11.12, Conjecture~D']{Gromov:MetricInequalitiesScalar}}]\label{conj:collar_width}
  Let \(X\) be a closed manifold of dimension \(n\) and such that \(X \setminus \{p_0\}\), \(p_0 \in X\), does not admit a complete metric of positive scalar curvature.
  Let \(M\) be the manifold with boundary obtained from \(X\) by removing an open ball around \(p_0\).
  Then for any Riemannian metric of scalar curvature \(\geq n(n-1) > 0\) on \(M\), the width of a geodesic collar neighborhood of \(\partial M\) is bounded above by \(\pi / n\).
\end{conj}

\begin{conj}[{\cite[11.12, Conjecture~C]{Gromov:MetricInequalitiesScalar}}]\label{conj:band-width}
  Let \(M\) be a closed connected manifold of dimension \(n-1\neq 4\) such that \(M\) does not admit a metric of positive scalar curvature.
  Let \(g\) be a Riemannian metric on \(V = M \times [-1,1]\) of scalar curvature bounded below by \(n(n-1) = \scal_{\Sphere^n}\).
  Then
  \[
    \width(V, g) \leq \frac{2\pi}{n},  
  \]
  where  \(\width(V,g) \coloneqq \dist_g(\partial_- V, \partial_+ V)\) is the distance between the two boundary components of \(V\) with respect to \(g\).
\end{conj}

Note that we have strengthened the bounds in \cref{conj:collar_width,conj:long_neck} compared to the original sources.
All these constants are optimal, as we shall discuss below.

Gromov's first definitive result on these questions~\cite{Gromov:MetricInequalitiesScalar} was a proof of \cref{conj:band-width} for the torus and related manifolds via the geometric measure theory approach to positive scalar curvature going back to the minimal hypersurface method of~\citeauthor{SchoenYau:HypersurfaceMethod}~\cite{SchoenYau:HypersurfaceMethod}.
In fact, it initially appeared that the Dirac operator technique was not suitable to such quantitative questions, in particular because they involve manifolds with boundary.
However, in recent articles of the authors~\cite{Zeidler:band-width-estimates,Cecchini:LongNeck,Zeidler:WidthLargeness}, we have demonstrated that the Dirac operator method can in principle be used to approach \cref{conj:band-width,conj:collar_width,conj:long_neck}.
A slightly different Dirac operator approach based on quantitative K-theory leading to similar (non-sharp) estimates for bands was subsequently given by~\citeauthor{Guo-Xie-Yu:Quantitative}~\cite{Guo-Xie-Yu:Quantitative}.

In the present article, we advance the spinor Dirac operator method further and put forward a novel point of view towards \cref{conj:band-width,conj:collar_width,conj:long_neck} which brings the mean curvature of the boundary into the focus of attention.
That is, under the assumption that the scalar curvature is bounded below by $n(n-1)$ and that suitable index-theoretic invariants do not vanish, we establish a precise quantitative relationship between the mean curvature of the boundary and the relevant distance quantity appearing in situations related to \cref{conj:band-width,conj:collar_width,conj:long_neck}.
More precisely, we show that in each case there exist constants $c_n(l) > 0$, depending on a distance parameter $l > 0$ and the dimension of the manifold, such that if the mean curvature is bounded below by $-c_n(l)$, then the relevant distance is at most \(l\).
The crucial property of these constants is that $c_n(l)\to\infty$ as $l$ approaches the conjectured distance bound.
In other words, as the relevant distance tends to the threshold, the mean curvature tends to \(-\infty\) somewhere at the boundary.
The geometric intuition behind this behavior is that the metric must collapse as the critical threshold is approached.
Moreover, our new point of view allows us to establish rigidity results for certain extremal cases of the predicted quantitative relationship between scalar curvature, mean curvature and distance.

On a technical level, one ingredient is to augment the spinor Dirac operator---similarly as in our previous approaches---by a potential defined in terms of a distance function.
This procedure modifies the classical Schrödinger--Lichnerowicz formula in a way that allows to relate distance estimates to spectral properties of the modified operator.
However, the crucial new ingredient is that we study a tailor-made boundary value problem associated to the augmented Dirac operator. 
This enables us to use spinorial techniques to not only quantitatively control the scalar curvature using a differential expression of the potential but also bring the mean curvature of the boundary into play.
The main principle behind our new approach is that we can compare certain spin manifolds to model spaces which are suitable warped products, provided that one can produce a non-trivial solution of a boundary value problem associated to the augmented Dirac operator on the given manifold.
Moreover, up to a constant, the potential directly corresponds to the mean curvature of the cross sections in the model warped product space.

We develop this approach in a general setting that allows to treat the results related to \cref{conj:band-width,conj:collar_width,conj:long_neck} as well as further novel results in an essentially unified way.
To this end, we introduce a new abstract geometric structure which we call a \emph{relative Dirac bundle}.
This is a Dirac bundle \(S \to M\) in the sense of Gromov and Lawson~\cite[Section~1]{GromovLawson:PSCDiracComplete} (see also \cite{LawsonMichelsohn:SpinGeometry}) together with a suitable bundle involution \(\RelDiracInv \in \Ct^\infty(M \setminus K, \End(S))\) which is defined outside a compact subset \(K \subset \interior{M}\) of the interior of the manifold; see~\cref{sec:relative Dirac bundles} for details.
The use of this structure is twofold: Firstly, together with a suitable function \(\psi \colon M \to \R\), it allows to define the potential term necessary for the precise quantitative estimates.
This leads to the \emph{Callias\footnote{We use this terminology because the study of Dirac operators with potential was initiated by Callias~\cite{Callias}.} operator}
\[
  \Callias_\psi = \Dirac + \psi \RelDiracInv,
  \] 
where \(\Dirac\) is the Dirac operator associated to the Dirac bundle \(S\).
Secondly, the involution \(\RelDiracInv\) can be used to define natural chiral boundary conditions, which are crucial for the development of a suitable index theory for relative Dirac bundles (see~\cref{sec:index-theory}) and allow spectral estimates for \(\Callias_\psi\) (see~\cref{sec:curvature-estimates}) relating the mean curvature with the value of the function \(\psi\) along the boundary.
These boundary conditions are related to the treatment of the cobordism theorem via a boundary value problem as in~\cite[Section~21]{Boos-Bavnbek-Wojciechowski:Elliptic-boundary-Dirac} and \cite[Section~6.3]{Baer-Ballmann:Guide-Boundary-value-problems-Dirac}.
Our chirality also allows for an auxiliary choice of sign for each boundary component reminiscent of the boundary conditions considered by \citeauthor{Freed:Index-thms-odd-dim}~\cite{Freed:Index-thms-odd-dim}.
This additional choice will be relevant in the proofs of our results related to \cref{conj:band-width}.

A further notable observation is that our construction has a vague formal similarity to \emph{\(\mu\)-bubbles} or \emph{generalized soap bubbles}, which have recently led to substantial advances via the geometric measure theory approach to scalar curvature, see~\cite[Section~5]{gromovFourLecturesScalar2019}, \cite{Zhu:WidthEstimates,Zhu:RigdityComplete,chodosh_li,Gromov:5d,LUY2021-PositiveMass}.
Indeed, the latter can be viewed as an augmentation of the minimal hypersurface method by suitable potentials.

In the following \cref{subsec:intro_length_of_neck,subsec:intro_bands,subsec:intro-rigidity} of the introduction, we present a simplified overview of our main results.
In the main body of the article, these are derived by working with a suitable relative Dirac bundle and choosing a potential appropriate for the situation at hand.

\subsection{Length of the neck}\label{subsec:intro_length_of_neck}
Here we present our main geometric results related to \cref{conj:collar_width,conj:long_neck}.
We improve the upper bound of \(\pi /n\) to an estimate depending on the mean curvature of the boundary.

In our first result we estimate the length of the neck of a Riemannian manifold with boundary. 
Recall that for a smooth map of Riemannian manifolds $\Phi\colon M\to N$, the \emph{area contraction constant} at \(p \in M\) is defined to be the norm of the induced map \(\Phi_\ast \colon \Lambda^2 \T_p M \to \Lambda^2 \T_{f(p)} N\) on \(2\)-vectors. We say the map is \emph{area non-increasing} if the area contraction constant is \(\leq 1\) at every point.
If \(M\) is compact and \(N\) is closed, both connected and oriented, where \(n = \dim M = \dim N \geq 2\), then a smooth map \(\Phi \colon M \to N\) that is locally constant near the boundary \(\partial M\) has a well-defined mapping degree \(\deg(\Phi) \in \Z\).\footnote{This can be defined as usual via counting the pre-image of a regular value with signs or, in cohomological terms, using the induced map \(\Z \cong \HZ^n(N; \Z) \cong \HZ^n(N, \Phi(\partial M); \Z) \xrightarrow{\Phi^\ast} \HZ^n(M, \partial M; \Z) \cong \Z\), where we use that \(\Phi(\partial M)\) is a finite set of points and \(n \geq 2\).}
Moreover, given a Riemannian manifold \((M,g)\), we denote the mean curvature of its boundary \(\partial M\) by \(\mean_g\) (or simply \(\mean\) if the metric is implicit); see also~\labelcref{eq:mean_convention} in~\cref{sec:relative Dirac bundles} for our sign- and normalization convention.

\begin{thm}[see \cref{sec:LongNeck}]\label{LongNeckProblem}
Let $(M,g)$ be a compact connected Riemannian spin manifold with non-empty boundary, $n=\dim M \geq 2$ even, and let $\Phi\colon M\to \Sphere^n$ be a smooth area non-increasing map.
Assume that \(\scal_g \geq n(n-1)\).
Moreover, suppose there exists $l\in \left(0,\pi/n\right)$ such that $\mean_g\geq-\tan(nl/2)$ and $\dist_g\bigl(\supp(\dd \Phi),\partial M\bigr) \geq l$ \parensup{compare \cref{fig:Long_neck}}.
Then $\deg(\Phi)=0$.
\end{thm}
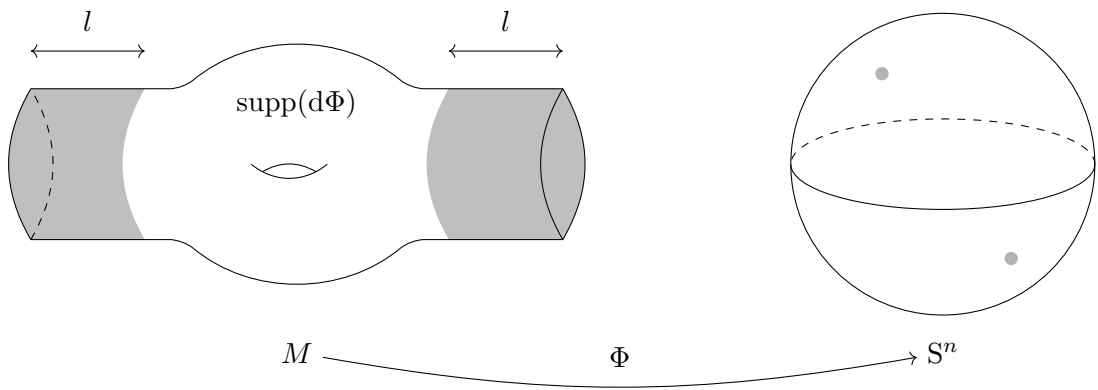
\begin{figure}[ht]
\begin{center}
\begin{tikzpicture}
\coordinate (A) at (0,-1);
\coordinate (B) at (0,1);
\coordinate (C) at (-7,1);
\coordinate (D) at (-7,-1);
\coordinate (M) at (5,0);

\fill [gray!50] (A) to [bend left=30] (B) to ++ (-1.5,0) to [bend right=30] ($(A) + (-1.5,0)$) -- cycle;
\draw[<->] ($(B) + (0,0.5)$) to node[label=above:{\(l\)}] {} ++ (-1.5,0);

\fill [gray!50] (D) to [bend left=30] (C) to ++ (1.5,0) to [bend right=30] ($(D) + (1.5,0)$) -- cycle;
\draw[<->] ($(C) + (0,0.5)$) to node[label=above:{\(l\)}] {} ++ (1.5,0);

\draw[fill=gray!50] (A) to [bend left=30] (B) to [bend left=30] (A);

\draw (D) to [bend left=30] (C);
\draw [dashed] (D) to [bend right=30] (C);

\draw [rounded corners=5] (A) to ++(-2,0) to [bend left=40]  ($ (D) + (2,0) $) to (D);
\draw [rounded corners=5] (B) to ++(-2,0) to [bend right=40]($ (C) + (2,0) $) to  (C);

\draw[bend left=40] (-3.1,-0) edge[name path=hd] (-4.1,-0);
\path[name path=hb] (-3.1,-0.2) to [bend right=40] (-4.1,-0.2);
\path[name intersections={of=hd and hb}];
\draw[bend right=30] (intersection-1) to (intersection-2);

\node at (-3.5,0.8) {\(\supp(\D \Phi)\)};
\node (Mlabel) at (-3.5,-2.5) {\(M\)};

\draw (M) circle (2);
\draw ($(M) + (-2,0)$) arc (180:360:2 and 0.6);
\draw[dashed] ($(M) + (2,0)$) arc (0:180:2 and 0.6);

\node[circle,fill=gray!60,minimum size=5pt,inner sep=0pt] (q1) at ($(M) + (-0.8,1.2)$) {};
\node[circle,fill=gray!60,minimum size=5pt,inner sep=0pt] (q2) at ($(M) + (0.9,-1.25)$) {};

\node (Spherelabel) at ($(M) + (0,-2.5)$) {\(\Sphere^n\)};

\draw[->] (Mlabel) to [bend right=10] node[label=above:{\(\Phi\)}] {} (Spherelabel);

\end{tikzpicture}
\caption{The long neck problem}\label{fig:Long_neck}
\end{center}
\end{figure}
The statement of~\cref{LongNeckProblem} is sharp; we discuss this in \cref{sec:LongNeck},~\cref{prop:long_neck_sharp}.
In this context, a subtle point hidden in the statement of the theorem is that we also rule out the equality situation \(\dist_g\bigl(\supp(\dd \Phi),\partial M\bigr) = l\) under these scalar- and mean curvature bounds if \(\deg(\Phi) \neq 0\).
This is in contrast to the situations of the other conjectures, where the equality situations can be realized, compare~\cref{rem:optimality} below.
Addressing this detail requires a considerably more precise analysis than in the earlier approach from~\cite[Theorem~A]{Ce20}.

\begin{cor}
Let $(M,g)$ be a compact connected Riemannian spin manifold with boundary, $n=\dim M \geq 2$ even, and let $\Phi\colon M\to \Sphere^n$ be a smooth area non-increasing map.
Assume that \(\scal_g \geq n(n-1)\).
If $\dist_g\bigl(\supp(\dd \Phi),\partial M\bigr) \geq \pi/n$, then \(\deg(\Phi) = 0\).
\end{cor}
This corollary is a direct consequence of of \cref{LongNeckProblem} because $\tan(nl/2)\to\infty$ as $l\to \pi/n$.
Thus this refines the original approach to~\cref{conj:long_neck} from \cite[Theorem~A]{Ce20}.
\Cref{prop:long_neck_sharp} also shows that the constant \(\pi/n\) is optimal here.

For our second result, we introduce the notion of $\Ahat$-area, which is a generalization of the notion of $\K$-area introduced in~\cite[\S 4]{Gromov:KArea}.
For a Hermitian bundle $E$, denote by $\FullCurv^E$ the curvature of the connection on $E$.

\begin{defi}\label{def:Ahat area}
    Let \((X,g)\) be a closed even-dimensional oriented Riemannian manifold.
    The \emph{\(\Ahat\)-area} of \((X, g)\) is the infimum of the numbers 
    \(
        \|\FullCurv^E\|_\infty^{-1},    
    \)
    ranging over all Hermitian bundles \(E\) with metric connections such that
    \(
        \int_X \AhatClass(X) \wedge \ch(E) \neq 0
    \), where \(\AhatClass(X)\) denotes the \(\Ahat\)-form of \(X\) and \(\ch(E)\) denotes the Chern character form of $E$.
\end{defi}
The \emph{\(\Ahat\)-area} of \((X, g)\) depends on the metric $g$.
However, since $X$ is compact, the notion of having \emph{infinite $\Ahat$-area} is independent of $g$.
  A closed spin manifold of infinite K-area also has infinite $\Ahat$-area.
  An important class of examples consists of even-dimensional compactly enlargeable manifolds, e.g.\ the $2m$-dimensional torus $\Torus^{2m}$.
  For the notion of enlargeability and more examples, we refer to \cite{gromov-lawson:spin-and-scalar-in-presence,GromovLawson:PSCDiracComplete}.
  More generally, if $X$ has infinite $\K$-area and $Y$ has nonvanishing $\Ahat$-genus, then $X\times Y$ has infinite $\Ahat$-area.
  This includes, for example, the Cartesian product $\Torus^{2m}\times Y$, where $Y$ is the K3-surface.

Now suppose $X$ is a closed $n$-dimensional enlargeable spin manifold.
By \cite[Theorem~C]{Cecchini:LongNeck}, \(X \setminus \{p_0\}\), \(p_0 \in X\), does not admit a complete metric of positive scalar curvature.
Moreover, the double of $M\coloneqq X\setminus \Ball^n$, with $\Ball^n$ an open $n$-ball embedded in $X$, is enlargeable as well.
Therefore, the next theorem in particular refines the upper bound of \cref{conj:collar_width} in the case of even-dimensional enlargeable spin manifolds using information from the mean curvature of the boundary.

\begin{thm}[see \cref{S:K-area}]\label{CollarInfiniteKArea}
Let $(M,g)$ be a compact connected $n$-dimensional Riemannian spin manifold with boundary such that the double of $M$ has infinite $\Ahat$-area.
Suppose that $\scal_g>0$ and that there exist positive constants $\kappa$ and $l$, with $0 < l < \pi/(\sqrt\kappa n)$, such that
\[
    \mean_g\geq -\sqrt{\kappa}\tan\left(\frac{\sqrt\kappa nl}{2}\right).
\]
Then the boundary \(\partial M\) admits no open geodesic collar neighborhood \(\mathcal{N} \subseteq M\) of width strictly greater than \(l\) such that \(\scal_g \geq \kappa n(n-1)\) on \(\mathcal{N}\).
\end{thm}
The estimate in this theorem is also optimal; see~\cref{rem:optimality} below.

\begin{rem}
Similarly as before, one deduces from this that the case $l\geq \pi/{(n\sqrt\kappa)}$ is ruled out independently of mean curvature restrictions. 
This also follows from the techniques in~\cite[Theorem~B]{Ce20} and see the discussion in~\cite{perspectives-in-psc:generalized-callias}.
\end{rem}

In particular, \cref{CollarInfiniteKArea} implies that a manifold with boundary whose double has infinite $\Ahat$-area cannot carry any metric of positive scalar curvature and mean convex boundary.
This also follows from recent results of  \citeauthor{BH20}~\cite{BH20}.

\subsection{Estimates of bands}\label{subsec:intro_bands}
Here we exhibit our main results related to \cref{conj:band-width}.
Similarly as above, we are able to improve the upper bound of \(2 \pi /n\) to a bound depending on the mean curvature of the boundary.
We will formulate our result for manifolds which are not only cylinders.
We say that a \emph{band} is a compact manifold \(V\) together with a decomposition \(\partial V = \partial_- V \sqcup \partial_+ V\), where \(\partial_\pm V\) are unions of components.
This notion goes back to~\cite{Gromov:MetricInequalitiesScalar}, where such manifolds are called compact proper bands.
A map \(V \to V'\) is a \emph{band map} if it takes \(\partial_\pm V\) to \(\partial_\pm V'\).
The \emph{width} \(\width(V,g)\) of a Riemannian band \((V,g)\) is the distance between \(\partial_{-}V\) to \(\partial_{+} V\) with respect to \(g\).

We now focus on a class of bands to which our results apply and is simple to describe. 
An \emph{overtorical band}~\cite[Section~2]{Gromov:MetricInequalitiesScalar} is a band \(V\) together with a smooth band map \(V \to \Torus^{n-1} \times [-1,1]\) of non-zero mapping degree, where \(\Torus^{n-1}\) denotes the torus of dimension \(n-1\).
More generally, we can also consider \(\Ahat\)-overtorical bands~\cite{Zeidler:WidthLargeness}, which are defined similarly but replacing the usual mapping degree by the \(\Ahat\)-degree in the sense of~\cite[Definition~2.6]{gromov-lawson:spin-and-scalar-in-presence}.

\begin{thm}[{cf.~\cref{cor:symmetric-band-estimate,cor:half-band-estimate}}] \label{thm:intro-band-estimate}
    Let \(n\) be odd and \((V,g)\) be an \(n\)-dimensional \(\Ahat\)-overtorical spin band.
    Suppose that \(\scal_g \geq n(n-1)\).
    If the mean curvature of the boundary satisfies
    \begin{itemize}
        \item  either \(\mean_g \geq -\tan(n l / 4)\) for some \(0 < l < 2\pi/n\),
        \item or \(\mean_g|_{\partial_- V} \geq 0\) and \(\mean_g|_{\partial_+ V} \geq -\tan(n l / 2)\) for some \(0 < l < \pi/n\),
    \end{itemize}
    then \(\width(V, g) \leq l\).
\end{thm}

Again, as the expression \(-\tan(n l / 4)\) tends to \(-\infty\) as \(l \to 2\pi/n\), we obtain the strict version of the original estimate desired by~\cref{conj:band-width}.
This also follows from~\cite[Corollary~1.5]{ZeidlerWidthLargeness}.
If, in addition, we assume that one of the boundary components is mean convex, then we can even obtain a strict bound of \(\pi/n\).

\begin{cor}[{cf.~\cref{cor:total-band-estimate}}]
    Let \(n\) be odd and \((V,g)\) be an \(n\)-dimensional \(\Ahat\)-overtorical spin band. 
    Suppose that \(\scal_g \geq n(n-1)\).
    Then we always have \(\width(V, g) < {2\pi}/{n}\).
    Moreover, if \(\partial_- V\) is mean convex, then \(\width(V, g) < {\pi}/{n}\).
\end{cor}

\begin{rem}
The statement of \cref{thm:intro-band-estimate} exhibited here in the introduction is a special case of the more general \cref{band-estimate} comparing the the mean curvature to arbitrary values of the form \(\mp \tan(n t_\pm /2)\), where \(-\pi/n < t_- < t_+ < \pi/n\), to get a corresponding width bound \(t_+ - t_-\).
\end{rem}
\begin{rem}
Our methods apply to a more general class of bands of \emph{infinite vertical \(\Ahat\)-area} which, in particular, includes bands diffeomorphic to \(M \times [-1,1]\) with \(M\) being a closed spin manifold of infinite \(\Ahat\)-area in the sense of \cref{def:Ahat area}. See~\cref{sec:bands} for details.
\end{rem}

\begin{rem}\label{rem:optimality}
  The band estimates given in \cref{thm:intro-band-estimate} are sharp.
  This follows from the warped product metric \(\varphi^2 g_{\Torus^{n-1}} + \D x \otimes \D x\) on \(\Torus^{n-1} \times [-l,l]\) for any \(0 < l < \pi/n\), where \(g_{\Torus^{n-1}}\) is the flat torus metric and \(\varphi(t) = \cos(nt/2)^{2/n}\), as indicated by Gromov in~\cite[653]{Gromov:MetricInequalitiesScalar}.
   By rescaling a given arbitrary metric \(g_M\) on any manifold \(M\) allowed in~\cref{conj:band-width}, the warped product metric \(\varphi^2 g_{M} + \D x \otimes \D x\) in fact shows optimality of this estimate on any band diffeomorphic to \(M \times [-1,1]\).
    Incidentally, this construction also shows optimality of the estimate in~\cref{CollarInfiniteKArea} by forgetting band structure and simply considering \(M = \Torus^{n-1} \times [-l,l]\) with this warped product metric.
\end{rem}

\subsection{Extremality and rigidity results}\label{subsec:intro-rigidity}
In view of the band width estimates, it is natural to investigate the extremal case.
The class of \(\Ahat\)-overtorical bands discussed above includes the special case of \emph{\(\Ahat\)-bands}, that is, bands such that \(\Ahat(\partial_- V) \neq 0\) (and thus, by bordism invariance also \(\Ahat(\partial_+ V) \neq 0\)).
In this special case, we prove the following rigidity theorem stating that the extremal case can only be achieved by the warped product construction discussed in~\cref{rem:optimality} over a Ricci flat manifold.

\begin{thm}[{cf.~\cref{cor:band-rigidity-concrete}}]\label{intro:band-rigidity}
Let \((V,g)\) be an \(n\)-dimensional band which is a spin manifold and satisfies \(\Ahat(\partial_- V) \neq 0\).
Suppose that \(\scal_g \geq n(n-1)\).
  Let \(0 < d < \pi/n\) and assume furthermore that one of the following conditions holds:
  \begin{itemize}
    \item either \(\width(V,g) \geq 2d\) and \(\mean_g|_{\partial V} \geq -\tan({n d}/2)\),
    \item or \(\width(V,g) \geq d\) and \(\mean_g|_{\partial_- V} \geq 0\), \(\mean_g|_{\partial_+ V} \geq -\tan({n d}/2)\).
  \end{itemize}
  Then \((V,g)\) is isometric to a warped product \((M \times I, \varphi^2 g_M + \D{x} \otimes \D{x})\), where either \(I = [-d,d]\) or \(I=[0,d]\), \(\varphi(t) = \cos\left(n t/{2}\right)^{2/n}\) and \(g_M\) is some Riemannian metric on \(M\) which carries a non-trivial parallel spinor.
  In particular, \(g_M\) is Ricci-flat.
\end{thm}
Again, we also have a more general version of this theorem involving arbitrary mean curvature bounds of the form \(\mp \tan(n t_\pm /2)\), see~\cref{thm:band-rigidity}.

The study of extremality questions about scalar curvature has a long history initiated by Gromov's K-area inequalities~\cite{Gromov:KArea}.
\citeauthor{LLarull}~\cite{LLarull} proved sharp inequalities using the Dirac operator which imply that the round metric on the sphere is scalar curvature extremal, that means, it cannot be enlarged without a decrease in the scalar curvature at some point.
Llarull's technique and results were subsequently refined and generalized by~\citeauthor{Goette-Semmelmann}~\cite{Goette-Semmelmann} and remain of central importance in contemporary research on scalar curvature, see for instance~\cite[Section~10]{Gromov:MetricInequalitiesScalar}, \cite[Section~4]{gromovFourLecturesScalar2019}, \cite{Zhang-SIGMA}.
\citeauthor{Lott:Boundary}~\cite{Lott:Boundary} recently extended the technique of Llarull and Goette and Semmelmann to even-dimensional manifolds with boundary using the Dirac operator with (local) boundary conditions.
The following results combine the technique of \citeauthor{Goette-Semmelmann} with our machinery to obtain a new kind of extremality result for a large class of warped product manifolds.

Let \(M\) be a manifold with boundary \(\partial M\).
We say that a Riemannian metric \(g_M\) on \(M\) is \emph{scalar-mean extremal} if every metric \(g\) on \(M\) which satisfies \(g \geq g_M\), \(\scal_{g} \geq \scal_{g_M}\) and \(\mean_g \geq \mean_{g_M}\) already satisfies \(\mean_{g_M} = \mean_g\) and \(\scal_{g_M} = \scal_{g}\).
Moreover, we say that \(g_M\) is \emph{scalar-mean rigid} if any such metric must satisfy \(g = g_M\).

\begin{thm}[{cf.~\cref{goette-semmelmann-band-extremality}}] \label{intro:goette-semmelmann-band-extremality}
    Let \(n\) be odd and \((M, g_M)\) be an \((n-1)\)-dimensional Riemannian spin manifold of non-vanishing Euler-characteristic whose Riemannian curvature operator is non-negative.
    Let \(\varphi \colon [t_-, t_+] \to (0,\infty)\) be a smooth strictly logarithmically concave function and consider the warped product metric \(g_V = \varphi^2 g_M + \D x \otimes \D x\) on \(V \coloneqq M \times [t_-, t_+]\).
    Then any metric \(g\) on \(V\) which satisfies 
    \begin{myenumi}
      \item \(g \geq g_V\),
      \item \(\scal_g \geq \scal_{g_V}\),
      \item \(\mean_g \geq \mean_{g_V}\)
    \end{myenumi}
    is itself a warped product \(g = \varphi^2 \tilde{g}_M + \D x \otimes \D x\) for some metric \(\tilde{g}_M\) on \(M\) such that \(\scal_{\tilde{g}_M} = \scal_{g_M}\).
    In particular, \(g_V\) is scalar-mean extremal.
    
    If, in addition, the metric \(g_M\) satisfies \(\Ric_{g_M} > 0\), then \(g_V\) is scalar-mean rigid.
\end{thm}
Note that strictly logarithmically concave means that \(\log(\varphi)'' < 0\).
Geometrically, this means that the mean curvature of \(M \times \{x\}\) in the warped product (with respect to the normal field \(\partial_x\)) is strictly increasing.

\begin{rem}
We have a more general version of this theorem involving distance non-increasing maps of non-zero degree to these spaces; see~\cref{thm:goette-semmelmann-band-rigidity}.
\end{rem}

A particularly interesting special case of manifolds which can be written as such log-concave warped products are annuli in spaces of constant curvature.
Indeed, as a direct consequence of \cref{intro:goette-semmelmann-band-extremality}, we obtain the following result.

\begin{cor}[cf.~\cref{annulus-extremal-rigidity}] \label{intro:annulus-extremal-rigidity}
    Let \(n \geq 3\) be odd and \((M_\kappa, g_\kappa)\) the \(n\)-dimensional simply connected space form of constant sectional curvature \(\kappa \in \R\).
    Let \(0 < t_- < t_+ < t_\infty\), where \(t_\infty = + \infty\) if  \(\kappa \leq 0\) and \(t_\infty = \pi / \sqrt{\kappa}\) if \(\kappa > 0\). Consider the annulus 
    \[\Annulus_{t_-, t_+} \coloneqq \{ p \in M_\kappa \mid t_- \leq d_{g_\kappa}(p,p_0) \leq t_+ \}\] 
    around some base-point \(p_0 \in M_\kappa\).
    Then the metric \(g_\kappa\) is scalar-mean rigid on \(\Annulus_{t_-, t_+}\).
  \end{cor}

Similar statements for log-concave warped products (and, in particular, punctured space forms) have been suggested by Gromov in \cite[Section~5.4]{gromovFourLecturesScalar2019} based on considerations with \(\mu\)-bubbles.
Moreover, different rigidity statements for hyperbolic space have been studied in the context of the positive mass theorem, see for instance~\cite{Min-Oo:ScalarCurvRigid,Andersson-Dahl:ScalarCurvRigid,Chrusciel-Herzlich:MassHyperbolic,Sakovich-Sormani:AlmostRigidHyperbolic}.

\begin{rem}
These results also appear quite similar to Lott's main result in~\cite{Lott:Boundary} (apart from the fact that they apply in complementary parities of the dimension).
They are, however, of an essentially different geometric nature. On the one hand, \cite{Lott:Boundary}~proves \emph{area} extremality rather than merely (length-) extremality as in our results.
Our technique here relies on the existence of a suitable Lipschitz function and does not appear to be readily applicable to the study of area non-increasing maps.
On the other hand, we do not require that the curvature operator of the metric on \(V\) itself is non-negative or that the second fundamental form of the boundary is non-negative.
Indeed, as the examples of hyperbolic annuli show, our results include manifolds of negative sectional curvature.
Similarly, suitable spherical annuli are examples which have negative second fundamental form at the boundary.
One should note that on a manifold \emph{without boundary}, a metric of negative scalar curvature can never be scalar extremal---simply rescaling using a constant \(> 1\) provides a counter-example.
However, in our example of an annulus in hyperbolic space, the outer boundary component has positive mean curvature which thwarts such a rescaling argument.
This shows that the presence of the outer boundary component is crucial for our result to hold in the negative curvature example.
More generally, in all the examples covered by \cref{intro:annulus-extremal-rigidity} it is the case that at least one quantity is positive among the scalar curvature in the interior and the mean curvatures on the two boundary components.
\end{rem}

\subsection{Higher index theory and future directions}
In the present article we are only working with classical Dirac-type operators on finite dimensional Hermitian vector bundles rather than using \emph{higher index theory} which would involve bundles with coefficients in infinite dimensional \textCstar-algebras.
This is in contrast to our previous work~\cite{Zeidler:band-width-estimates,Ce20,Zeidler:WidthLargeness}.
The main rationale behind this change of perspective is that it allows a quicker and more accessible exposition of the novel geometric arguments we want to exhibit through this article.
At the same time, this does not sacrifice too much of the possible generality of the statements, because many examples which are usually approached via higher index theory can already be dealt with via more classical notions like \emph{enlargeability} or \emph{infinite \(\K\)-area}.
It does, however, restrict the parity of the dimension in some of the results.

Note that our central structure of a \emph{relative Dirac bundle} can be straightforwardly generalized to \textCstar-algebra coefficients.
In this way it would be possible to reformulate all arguments from~\cite{Zeidler:band-width-estimates,Ce20,Zeidler:WidthLargeness} in terms of this concept.
This will be partly explained in~\cite{perspectives-in-psc:generalized-callias}.
Since this notion is fairly general, we also expect that it can be applied in many other geometric contexts we have not considered thus far.
 However, extending the results from this article to higher index theory requires some new analytic work because the present state of the art in the literature on boundary value problems for Dirac-type operators, where we mostly rely on~\citeauthor{Baer-Ballmann:Boundary-value-problems-first-order}~\cite{Baer-Ballmann:Guide-Boundary-value-problems-Dirac,Baer-Ballmann:Boundary-value-problems-first-order}, does not cover the case of infinite dimensional bundles.
Since this aspect is orthogonal to the geometric arguments presented here, we will address it separately in future work.

\paragraph*{Acknowledgments.}
We thank Misha Gromov for useful comments on an earlier version of this manuscript pointing us to the examples showing optimality of \cref{LongNeckProblem}.
We also would like to thank Bernd Ammann for a brief helpful correspondence and the anonymous referee for valuable suggestions improving the exposition.
\section{Relative Dirac bundles}\label{sec:relative Dirac bundles}
In this section, we set up the differential geometric preliminaries underlying the rest of this article.
In particular, we introduce our new concept of a \emph{relative Dirac bundle}.

We begin by fixing some notation.
Let $(M,g)$ be a Riemannian manifold and let $E\to M$ be a Hermitian vector bundle.
The space of smooth sections will be denoted by \(\Ct^\infty(M,E)\) and the subspace of compactly supported smooth sections by $\Ct^\infty_\cpt(M,E)$.
We denote fiberwise inner products by $\langle\blank,\blank\rangle$ and fiberwise norms by $|\blank|$.
Next we recall the notion of a classic Dirac bundle in the sense of Gromov and Lawson.
\begin{defi}[{\cite[Section~1]{GromovLawson:PSCDiracComplete}, \cite{LawsonMichelsohn:SpinGeometry}}]
  A (\(\Z/2\)-graded) \emph{Dirac bundle} over $M$ is a Hermitian vector bundle $S\to M$ with a metric connection \(\nabla \colon \Ct^\infty(M,S) \to \Ct^\infty(M,\T^\ast M \otimes S)\) (endowed with a parallel and orthogonal \(\Z/2\)-grading \(S = S^+ \oplus S^-\)) and a parallel bundle map
  \(
    \clm \colon \T^\ast M \to \End(S)
  \), called \emph{Clifford multiplication},
  such that \(\clm(\omega)\) is anti-self-adjoint (and odd), and \(\clm(\omega)^2 = - |\omega|^2\) for all \(\omega \in \T^\ast M\).
\end{defi}
For simplicity, we sometimes denote Clifford multiplication by \(\omega \cdot u \coloneqq \clm(\omega) u\).
Let $\Dirac = \sum_{i=1}^n \clm(e^i)\nabla_{e_i} \colon\Ct^\infty(M,S)\to\Ct^\infty(M,S)$ be the associated Dirac operator, where \(e_1, \dotsc, e_n\) is a local frame and \(e^1, \dotsc, e^n\) the dual coframe.
If we have a $\Z/2$-grading $S=S^+\oplus S^-$, the operator $\Dirac$ is odd, that is, it is of the form
\[
    \Dirac=\begin{pmatrix}
    0&\Dirac^-\\\Dirac^+&0
    \end{pmatrix},
\]
where $\Dirac^\pm\colon\Ct^\infty(M,S^\pm)\to\Ct^\infty(M,S^\mp)$ are formally adjoint to one another.

We now turn to relative Dirac bundles, an augmentation of a Dirac bundle.
\begin{defi}\label{D:relative Dirac bundle}
  Let \(K \subset \interior{M}\) be compact subset in the interior.
  A \emph{relative Dirac bundle} with \emph{support} \(K\) is a \(\Z/2\)-graded Dirac bundle \(S \to M\) together with an odd, self-adjoint, parallel bundle involution \(\RelDiracInv \in \Ct^\infty(M \setminus K, \End(S))\) satisfying \(\clm(\omega) \RelDiracInv = - \RelDiracInv \clm(\omega)\) for every \(\omega \in \T^\ast M|_{M \setminus K}\) and such that \(\RelDiracInv\) admits a smooth extension to a bundle map on an open neighborhood of $\overline{M\setminus K}$.
\end{defi}
The final technical requirement in particular ensures the existence of a unique continuous extension of \(\RelDiracInv\) to the topological boundary of \(K\).
We do not consider the further extension to a neighborhood of \(\overline{M\setminus K}\) part of the data, we just require its existence.

It follows directly from the definition that the Dirac operator anti-commutes with the involution \(\RelDiracInv\) where it is defined.

One main use of this structure is that it will allow us to associate local boundary conditions to any choice of a sign for each connected component of $\partial M$.
\begin{defi}[Boundary chirality for relative Dirac bundles]\label{D:BoundaryChirality}
  For a relative Dirac bundle \(S \to M\) and a locally constant function \(s \colon \partial M \to \{\pm 1\}\),
  we say that the endomorphism
  \begin{equation}
    \chi \coloneqq s \clm(\nu^\flat) \RelDiracInv \colon S|_{\partial M} \to S|_{\partial M} \label{eq:boundary-involution}
  \end{equation}
  is the \emph{boundary chirality} on $S$ associated to choice of signs $s$.
  Here, \(\nu \in \Ct^\infty(\partial M, \T M|_{\partial M})\) is the inward pointing unit normal field.
\end{defi}
Note that \(\chi\) is a self-adjoint even involution which anti-commutes with \(\clm(\nu^\flat)\) but commutes with \(\clm(\omega)\) for all \(\omega \in \T^\ast(\partial M)\).
  This defines local boundary conditions for sections \(u \in \Ct^\infty(M, S)\) by requiring that
  \begin{equation}
    \chi(u|_{\partial M}) = u|_{\partial M}.
    \label{eq:boundary-conditions}
  \end{equation}
In Section~\ref{sec:index-theory}, we will observe that these boundary conditions are elliptic.
The next lemma will be used to show that they are self-adjoint.

\begin{lem} \label{lem:Lagrangian}
  Let \(S \to M\) be a relative Dirac bundle and \(s \colon \partial M \to \{\pm 1\}\) a choice of signs.
  Then \(L(\chi) \coloneqq \{ u \in S|_{\partial M} \mid \chi(u) = u\}\) is a Lagrangian subbundle of \(S|_{\partial M}\) with respect to the fiberwise symplectic form \((u,v) \mapsto \langle u, \clm(\nu^\flat) v \rangle\).
  In other words, \(\clm(\nu^\flat) L(\chi) = L(\chi)^\perp\).
\end{lem}
\begin{proof}
  Note that \(L(\chi)\) is the image of the orthogonal bundle projection 
  \((1 + \chi)/2\).
  Since \(L(\chi)^\perp\) is the image of the complementary projection \((1 - \chi)/2\) and \(\chi\) anti-commutes with \(\clm(\nu^\flat)\), we obtain \(\clm(\nu^\flat) L(\chi) = L(\chi)^\perp\).
\end{proof}

In the following, we recall several standard formulas on Dirac bundles for later use.
\begin{itemize}
  \item \emph{Green's formula}~(\cite[Proposition~9.1]{Taylor:PDE1}). For the Dirac operator,
  \begin{equation}
    \int_M \langle \Dirac u, v \rangle\ \vol_M = \int_M \langle u, \Dirac v \rangle\ \vol_M + \int_{\partial M} \langle u, \nu^\flat \cdot v \rangle\ \vol_{\partial M}, \label{eq:Dirac-Green}
  \end{equation}
  and for the connection,
  \begin{equation}
    \int_M \langle \nabla u, \varphi \rangle\ \vol_M = \int_M \langle u, \nabla^\ast \varphi \rangle\ \vol_M - \int_{\partial M} \langle u, \varphi(\nu) \rangle\ \vol_{\partial M},
    \label{eq:connection-Green}
  \end{equation}
  where \(u, v \in \Ct^\infty_\cpt(M,S)\), \(\varphi \in \Ct^\infty_\cpt(M, \T^\ast M \otimes S)\) and \(\nu \in \Ct^\infty(\partial M, \T M|_{\partial M})\) is the inward-pointing normal field.
  \item \emph{Bochner--Lichnerowicz--Weitzenböck formula}~(\cite[Proposition~2.5]{GromovLawson:PSCDiracComplete}).
  \begin{equation}\label{Weitzenbock}
    \Dirac^2 = \nabla^\ast \nabla + \Curv,
  \end{equation}
  where \(\nabla^\ast \nabla = -\sum_{i=1}^n \nabla^2_{e_i, e_i}\) is the connection laplacian and 
  \[\Curv = \sum_{i < j} \clm(e^i) \clm(e^j) \mathrm{R}^\nabla(e_i, e_j)\]
  is a bundle endomorphism linearly depending on the curvature tensor \(\mathrm{R}^\nabla\) of \(\nabla\).
  \item \emph{Penrose operator and Friedrich inequality}~(\cite[Section~5.2]{SpinorialApproach}). We have that
  \begin{equation}
    | \nabla u |^2 = | \Penrose u |^2 + \frac{1}{n} | \Dirac u |^2 \label{eq:Penrose-Dirac}
  \end{equation}
  for all \(u \in \Ct^\infty(M, S)\), where \(\Penrose \colon \Ct^\infty(M,S) \to \Ct^\infty(M,\T^\ast M \otimes S)\) is the \emph{Penrose operator} defined as
  \[
    \Penrose_\xi u \coloneqq \nabla_\xi u + \frac{1}{n} \xi^\flat \cdot \Dirac u.
  \]
  In particular, we have the \emph{Friedrich inequality}
  \[
    |\nabla u|^2 \geq \frac{1}{n} |\Dirac u|^2,
  \]
  where equality holds if and only if \(\Penrose u = 0\), that is, \(u\) satisfies the \emph{twistor equation} 
  \begin{equation}
    \nabla_\xi u + \frac{1}{n} \xi^\flat \cdot \Dirac u = 0 \label{eq:twistor}
  \end{equation}
  for all \(\xi \in \T M\).
\item \emph{Boundary Dirac operator}~(\cite[Appendix~1]{Baer-Ballmann:Guide-Boundary-value-problems-Dirac}, \cite[Introduction]{Hijazi-Montiel-Roldan:Eigenvalue-boundary-Dirac}).
Let \(\nu\) be the interior unit normal field.
We turn \(S^\partial = S|_{\partial M}\) into a Dirac bundle via the Clifford multiplication and connection,
\begin{align}
  \clm^\partial(\omega) = \clm(\omega) \clm(\nu^\flat), \quad \omega \in \T^\ast \partial M, \nonumber\\
  \nabla^\partial_\xi =\nabla_\xi +\frac{1}{2}\clm^{\partial}(\nabla_\xi \nu^\flat), \quad \xi \in \T \partial M.\quad \label{eq:boundary-connection}
\end{align}
Let 
\begin{equation}
\BdDirac \colon \Ct^\infty(\partial M,S^\partial) \to \Ct^\infty(\partial M,S^\partial), \quad \BdDirac = \sum_{i=1}^{n-1} \clm^\partial(e^i) \nabla^\partial_{e_i} \label{eq:canonical-boundary-operator}
\end{equation} denote the corresponding Dirac operator.
This is the \emph{canonical boundary operator} associated to the Dirac bundle \(S \to M\).
Note that \(\BdDirac\) is even with respect to the grading on \(S^\partial\) restricted from the grading on \(S\).
It satisfies
\begin{equation}
  \BdDirac \clm(\nu^\flat) = -\clm(\nu^\flat) \BdDirac,
\end{equation}
and, if \(S \to M\) is endowed with the structure of a relative Dirac bundle, then
\begin{equation}
  \BdDirac \RelDiracInv = \RelDiracInv \BdDirac, \qquad \chi \BdDirac = - \BdDirac \chi, \label{eq:chirality-anti-commutes}
\end{equation}
where \(\chi\) is defined as in \labelcref{eq:boundary-involution} with respect to any choice of signs.
This implies that, if \(u \in \Ct^\infty(M, S)\) satisfies the boundary condition~\labelcref{eq:boundary-conditions}, then
\begin{equation}
  \langle u|_{\partial M}, \BdDirac u|_{\partial M} \rangle = 0.
  \label{eq:BdDirac-vanishes-under-bd-condition}
\end{equation}
Moreover, 
\begin{align} \BdDirac  &= \frac{n-1}{2} \mean  - \clm(\nu^\flat)\sum_{i=1}^{n-1}\clm(e^i) \nabla_{e_i}  \nonumber \\
  &=\frac{n-1}{2} \mean  - \clm(\nu^\flat) \Dirac - \nabla_\nu , \label{eq:boundary-dirac-vs-mean}
  \end{align}
  where \(\mean\) is the mean curvature of \(\partial M\) with respect to \(\nu\).
  To avoid any confusion about signs and normalization, let us be explicit about our convention for the mean curvature:
\begin{equation}
\mean = \frac{1}{n-1} \tr (-\nabla \nu) = \frac{1}{n-1} \sum_{i=1}^{n-1} \langle e_i, - \nabla_{e_i} \nu \rangle = \frac{1}{n-1} \sum_{i=1}^{n-1} \II(e_i, e_i). \label{eq:mean_convention}
\end{equation}
Here \(-\nabla \nu\) is the shape operator, \(\II\) denotes the second fundamental form and we use a local orthonormal frame \(e_1, \dotsc, e_{n-1}\) on \(\partial M\).
\end{itemize}

Finally, in the remainder of this section, we discuss two concrete geometric examples of relative Dirac bundles which are relevant for our main results.
\begin{ex} \label{ex:Gromov-Lawson-Relative}
  Let \(M\) be a compact even-dimensional Riemannian spin manifold with boundary and let $\slashed S_M=\slashed S_M^+\oplus \slashed S_M^-$ be the associated $\Z/2$-graded complex spinor bundle endowed with the Levi-Civita connection.
  Let \(E, F \to M\) be a pair of Hermitian bundles equipped with metric connections.
  Note that the bundle
  \[
  S \coloneqq \slashed S_M \tensgr \bigl(E \oplus F^{\op}\bigr)
  \]
  is a $\Z/2$-graded Dirac bundle, where the grading $S=S^+\oplus S^-$ is given by
  \[
  S^+\coloneqq\bigl( \slashed S_M^+\tensor E\bigr)\oplus\bigl(\slashed S_M^-\tensor F\bigr)\qquad\text{and}\qquad
  S^-\coloneqq\bigl(\slashed S_M^+\tensor F\bigr)\oplus\bigl(\slashed S_M^-\tensor E\bigr).
  \]
  In analogy with Gromov and Lawson~\cite{GromovLawson:PSCDiracComplete}, we make the following assumption.
  \begin{texteqn}\label{GromovLawsonAssumption}
  There exists a compact set \(K \subset \interior{M}\) and a parallel unitary bundle isomorphism \(\mathfrak{t} \colon E|_{M \setminus K} \to F|_{M \setminus K}\) which extends to a smooth bundle map on a neighborhood of \(\overline{M \setminus K}\).
  \end{texteqn}
  In this case, we say that $(E,F)$ is a \emph{GL pair} with \emph{support $K$}.
  Note that Condition~\eqref{GromovLawsonAssumption} implies that \(S\) is a relative Dirac bundle with involution
  \[
   \RelDiracInv \coloneqq \id_{\ReducedSpinBdl_M} \tensgr \begin{pmatrix}0 & \mathfrak{t}^\ast \\ \mathfrak{t} & 0 \end{pmatrix} \colon S|_{M \setminus K} \to S|_{M \setminus K},
  \]
  where \enquote{\(\tensgr\)} stands for the graded tensor product of operators, compare~\cite[Appendix~A]{Higson-Roe:KHomology}.
  This means \(\RelDiracInv((u_+ \oplus u_-) \otimes (e \oplus f)) = (u_+ \oplus - u_-) \otimes (\mathfrak{t}^\ast f \oplus \mathfrak{t} e)\) for sections \(u_\pm \in \Ct^\infty(M, \ReducedSpinBdl_M^\pm)\), \(e \in \Ct^\infty(M, E)\), \(f \in \Ct^\infty(M, F)\).
    The Dirac operator on $S$ is described as follows.
  Let $\ReducedSpinDirac_E$ and $\ReducedSpinDirac_F$ be the operators obtained by twisting the complex spin Dirac operator on $M$ respectively with the bundles $E$ and $F$.
  Observe that $\ReducedSpinDirac_E$ and $\ReducedSpinDirac_F$ are odd operators, that is, they are of the form
  \[
    \ReducedSpinDirac_E=\begin{pmatrix}0&\ReducedSpinDirac_E^-\\\ReducedSpinDirac_E^+&0\end{pmatrix}\qquad\text{and}\qquad
    \ReducedSpinDirac_F=\begin{pmatrix}0&\ReducedSpinDirac_F^-\\\ReducedSpinDirac_F^+&0\end{pmatrix},
  \]
where $\ReducedSpinDirac_E^+\colon\Cc^\infty(M,\slashed S^+\tensor E)\to \Cc^\infty(M,\slashed S^-\tensor E)$, 
$\ReducedSpinDirac_F^+\colon\Ct^\infty(M,\slashed S^+\tensor F)\to \Ct^\infty(M,\slashed S^-\tensor F)$ and $\ReducedSpinDirac_E^-$, $\ReducedSpinDirac_F^-$ are formally adjoint respectively to $\ReducedSpinDirac_E^+$, $\ReducedSpinDirac_F^+$.
The Dirac operator on $S$ is given by
\begin{equation}
    \Dirac=\begin{pmatrix}
    0&\Dirac^-\\\Dirac^+&0
    \end{pmatrix}\colon\Ct^\infty(M,S)\to\Ct^\infty(M,S),
\end{equation}
where $\Dirac^\pm\colon\Ct^\infty(M,S^\pm)\to\Ct^\infty(M,S^\mp)$ are the operators defined as
\begin{equation}\label{E:LNP1}
	\Dirac^+\coloneqq
	\begin{pmatrix}0 &\ReducedSpinDirac_F^- \\ \ReducedSpinDirac_E^+ & 0 	\end{pmatrix}
	\qquad\text{and}\qquad
		\Dirac^-\coloneqq
	\begin{pmatrix}0 &\ReducedSpinDirac_E^- \\ \ReducedSpinDirac_F^+ & 0 	\end{pmatrix}.
\end{equation}
Moreover, the curvature endomorphism~\(\Curv\) from \labelcref{Weitzenbock} is given by
\begin{equation}\label{E:GL-Lichnerowicz}
    \Curv = \frac{\scal_g}{4}+\Curv^{E\oplus F},
\end{equation}
 where $\mathcal{R}^{E\oplus F} = \sum_{i < j} \clm(e^i)\clm(e^j) (\id_{\ReducedSpinBdl_M} \tensor \FullCurv^{\nabla^{E \oplus F}}_{e_i, e_j})$ is an even endomorphism of the bundle $S$ which depends linearly on the curvature of the connection on $E \oplus F$; compare~\cite[Theorem~8.17]{LawsonMichelsohn:SpinGeometry}.
\end{ex}

\begin{ex} \label{ex:SimpleBand}
  Let $(V,g)$ be an odd-dimensional Riemannian spin band and let \(\ReducedSpinBdl_V \to V\) be the associated complex spinor bundle endowed with the connection induced by the Levi-Civita connection.
  Let \(E \to M\) be a Hermitian bundle equipped with a metric connection.
  Then \(S \coloneqq \bigl(\ReducedSpinBdl_V\tensor E\bigr) \oplus \bigl(\ReducedSpinBdl_V\tensor E\bigr)\) is a $\Z/2$-graded Dirac bundle with Clifford multiplication
  \[\clm \coloneqq \begin{pmatrix} 0 & \clm_{\ReducedSpinBdl}\tensor \id_E \\ \clm_{\ReducedSpinBdl}\tensor \id_E & 0 \end{pmatrix},\]
  where \(\clm_{\ReducedSpinBdl}\) is the Clifford multiplication on \(\ReducedSpinBdl_V\).
  Moreover, $S$ turns into a relative Dirac bundle with involution
  \begin{equation}\label{E:GLepsilon}
    \RelDiracInv \coloneqq \begin{pmatrix} 0 & -\iu \\ \iu & 0 \end{pmatrix} 
  \end{equation}
  globally defined on $V$ (that is, the support is empty).
  The Dirac operator on $S$ is given by
  \[
    \Dirac = \begin{pmatrix} 0 & \ReducedSpinDirac_E \\ \ReducedSpinDirac_E & 0 \end{pmatrix},
  \]
  where $\ReducedSpinDirac_E\colon\Ct^\infty(V,\slashed S_V\tensor E)\to \Ct^\infty(V,\slashed S_V\tensor E)$ is the spinor Dirac operator on $(V,g)$ twisted with the bundle $E$.
  As in the previous example, the curvature term from~\labelcref{Weitzenbock} is of the form
  \begin{equation}\label{E:Band-Lichnerowicz}
    \Curv =  \frac{\scal_g}{4}+\Curv^{E},
  \end{equation}
  where \(\Curv^E = \sum_{i < j} \clm(e^i) \clm(e^j) (\id_{\ReducedSpinBdl_V \oplus \ReducedSpinBdl_V} \otimes \FullCurv^{\nabla^E}_{e_i, e_j})\).
\end{ex}
\section{Index theory for relative Dirac bundles} \label{sec:index-theory}
In this section, we introduce the Callias operators associated to relative Dirac bundles and review the necessary analysis to develop an index theory for them.

We again start by briefly fixing the notation we are going to use.
Let \(E \to M\) be a Hermitian vector bundle over a smooth manifold \(M\) with \emph{compact} boundary \(\partial M\).
The $\Lp^2$-inner product of two sections $u$, $v\in \Ct^\infty_\textrm{c}(M,E)$ is defined by
\[
    (u,v) \coloneqq\int_M\left<u,v\right>\vol_M.
\]
The corresponding \(\Lp^2\)-norm will be denoted by \(\|u\| \coloneqq (u,u)^{1/2}\).
The space of \emph{square-integrable sections}, denoted by \(\Lp^2(M, E)\), can be identified with the completion of \(\Ct^\infty_\cpt(M,E)\) with respect to this norm.
We denote the space of \emph{locally square-integrable sections} by \(\Lp^2_\loc(M,E)\).
Similarly, the latter comes endowed with the family of semi-norms \(\|\blank\|_K\) ranging over compact subsets \(K \subseteq M\), where \(\|u\|_K \coloneqq \left(\int_K \langle u, u \rangle \vol_M \right)^{1/2}\).

We will also use Sobolev spaces to a limited extent; see for instance~\cite[Chapter~4]{Taylor:PDE1} for a detailed introduction.
The Sobolev space \(\SobolevH^1_{\loc}(M, E)\) consists of all sections \(u\in\Lp^2_\loc(M,E)\) such that \(\nabla u\), a priori defined as a distributional section over the interior of \(M\), is also represented by a locally square integrable section.
The space \(\SobolevH^1_{\loc}(M, E)\) is topologized using the family of semi-norms \(\|\blank\|_{\SobolevH^1_K}\) defined by \(\|u\|_{\SobolevH^1_{K}}^2 \coloneqq \|u\|^2_K + \|\nabla u\|_K^2\), where \(K\) ranges over all compact subsets of \(M\).
Since the boundary is assumed to be compact, the restriction \(\Ct^\infty_\cpt(M, E) \to \Ct^\infty(\partial M, E|_{\partial M})\), \(u \mapsto u|_{\partial M}\) extends to a continuous linear operator \(\tau \colon \SobolevH^1_\loc(M, E) \to \SobolevH^{1/2}(\partial M, E|_{\partial M})\) by the trace theorem; see e.g.~\cite[Chapter~4, Proposition~4.5]{Taylor:PDE1}.
Here \(\SobolevH^{1/2}(\partial M, E|_{\partial M})\) denotes a fractional Sobolev space for instance in the sense of~\cite[Chapter~4, Section~3]{Taylor:PDE1}.
An equivalent (and in our setup more relevant) description in the presence of a suitable Dirac operator on the boundary is given in~\cite[Section~4.1]{Baer-Ballmann:Guide-Boundary-value-problems-Dirac}.

Next we turn to the main player, the Callias operator associated to a relative Dirac bundle and a suitable potential function.
To this end, fix a complete Riemannian manifold $M$ with compact boundary \(\partial M\) and \(S \to M\) a relative Dirac bundle with support $K$.
This means that $K\subset \interior M$ is compact and the involution \(\RelDiracInv\) is defined on \(M \setminus K\); see \cref{D:relative Dirac bundle}.
\begin{defi}
A Lipschitz function \(\psi \colon M \to \R\) is called an \emph{admissible potential} if \(\psi=0\) on \(K\) and there exists a compact set $K \subseteq L \subseteq M$ such that $\psi$ is equal to a nonzero constant on each component of $M\setminus L$.
\end{defi}
Let $\psi$ be an admissible potential.
Then \(\psi \RelDiracInv\) extends by zero to a continuous bundle map on all of \(M\).
The \emph{Callias operator} associated to these data is the differential operator 
\begin{equation}
  \Callias_{\psi} \coloneqq \Dirac + \psi \RelDiracInv.  \label{eq:defi_callias}
\end{equation}
Since \(\psi\) is Lipschitz, the commutator $[\Dirac,\psi]$ extends to a bounded operator on \(\Lp^2(M, S)\) (compare e.g.~\cite[Lemma~3.1]{Baer-Ballmann:Boundary-value-problems-first-order}).
In fact, by Rademacher's theorem~\cite{Rademacher:Lipschitz}, the differential \(\D \psi\) exists almost everywhere and is an \(\Lp^\infty\)-section of the cotangent bundle.
In this view, the commutator is given by the formula $[\Dirac,\psi]=\clm(\D \psi)$ as an element of \(\Lp^\infty(M,\End(S))\).
Using that the involution \(\RelDiracInv\) anti-commutes with the Dirac operator, a direct computation then yields
\begin{equation}
  \Callias_{\psi}^2 = \Dirac^2 + \clm(\D \psi) \RelDiracInv + \psi^2.\label{eq:square_of_callias}
\end{equation}
To be very precise, this expression implicitly uses the requirement in~\cref{D:relative Dirac bundle} that \(\RelDiracInv\) extends to a smooth section on a neighborhood of \(\overline{M \setminus K}\).
However, as \(\D \psi = 0\) on the interior of \(K\), we only need the values of the continuous extension of \(\RelDiracInv\) to \(\overline{M \setminus K}\) to specify \(\clm(\D \psi) \RelDiracInv\), and this is only relevant if the topological boundary of \(K\) has positive Lebesgue measure.

Observe that the distance function $x$ from any fixed subset of $M$ is $1$-Lipschitz so that $|\D x|\leq 1$ almost everywhere.
In our applications, we will consider potentials of the form $\psi=Y\circ x$, where $Y$ is a smooth function on $\R$ and $x$ is a distance function from a geometrically relevant compact region of $M$.

For a choice of signs \(s \colon \partial M \to \{\pm 1\}\), we let \(\Callias_{\psi, s}\) denote the operator \(\Callias_\psi\) on the domain
\[
  \Css^\infty(M,S)\coloneqq  \{u \in \Cc^\infty(M, S) \mid \chi(u|_{\partial M}) = u|_{\partial M}\},
\]
where \(\chi \coloneqq s \clm(\nu^\flat) \RelDiracInv \) as in \labelcref{eq:boundary-involution}.

Note that, by definition, \(S = S^+ \oplus S^-\) is \(\Z/2\)-graded and the differential operator \(\Callias_{\psi}\) can be written as
\[
  \Callias_{\psi} =  \begin{pmatrix} 0 & \Callias_{\psi}^- \\ \Callias_{\psi}^+ & 0 \end{pmatrix},
\]
where \(\Callias_{\psi}^\pm\) are differential operators \(\Ct^\infty(M, S^\pm) \to \Lp^2(M, S^\mp)\).
Since \(\chi\) is even with respect to the grading, \(\Callias_{\psi, s}\) also decomposes similarly into operators \[\Callias_{\psi, s}^\pm \colon \{u \in \Cc^\infty(M, S^\pm) \mid \chi (u|_{\partial M}) = u|_{\partial M}\} \to \Lp^2(M, S^\mp).\]
Note that---even ignoring regularity---\(\Callias_{\psi, s}\) does \emph{not} preserve its domain.

In the following, we will discuss the necessary analysis and index theory for these operators.
We will mostly rely on the general framework of elliptic boundary value problems for Dirac-type operators due to \citeauthor{Baer-Ballmann:Guide-Boundary-value-problems-Dirac}~\cite{Baer-Ballmann:Guide-Boundary-value-problems-Dirac,Baer-Ballmann:Boundary-value-problems-first-order}.
However, as we will need to allow potentials which are a priori only Lipschitz for some of our applications, we will take a slight detour and first apply the results of \citeauthor{Baer-Ballmann:Guide-Boundary-value-problems-Dirac} only to the Dirac operator \(\Dirac\) and Callias operators \(\Callias_{\psi}\) with smooth potentials before deducing the desired statements for the general case.

We first observe that the canonical boundary Dirac operator \(\BdDirac\) from \labelcref{eq:canonical-boundary-operator} is an adapted operator for \(\Dirac\) in the sense of \cite[Section~3.2]{Baer-Ballmann:Guide-Boundary-value-problems-Dirac}.
  Since \(\chi\) anti-commutes with \(\BdDirac\) (see \labelcref{eq:chirality-anti-commutes}), \(\chi\) is a \emph{boundary chirality} in the sense of \cite[Exampe~4.20]{Baer-Ballmann:Guide-Boundary-value-problems-Dirac}.
  Thus \(B_\chi \coloneqq \SobolevH^{1/2}(\partial M, S_{\partial M}^{\chi})\), where \(S_{\partial M}^{\chi}\) is the \(+1\)-eigenbundle of \(\chi\), is an elliptic boundary condition in the sense of \citeauthor{Baer-Ballmann:Guide-Boundary-value-problems-Dirac}.
  Moreover, it is a consequence of \cref{lem:Lagrangian} that the adjoint boundary condition of \(B_\chi\) (see \cite[Theorem~4.6]{Baer-Ballmann:Guide-Boundary-value-problems-Dirac}, \cite[Section~7.2]{Baer-Ballmann:Boundary-value-problems-first-order}) is \(B_\chi\) itself (the crucial ingredient here is that \(\chi\) anti-commutes with \(\clm(\nu^\flat)\)).
  In other words, \(B_\chi\) is a self-adjoint boundary condition, and thus \(\Dirac\) is essentially self-adjoint on the domain \(\Css^\infty(M,S)\) by \cite[Theorem~4.11]{Baer-Ballmann:Guide-Boundary-value-problems-Dirac}.
  Moreover, \cite[Lemma~7.3]{Baer-Ballmann:Boundary-value-problems-first-order} implies that the domain of the closure is given by the Sobolev space
\[
  \SobolevH^1_{\RelDiracInv, s}(M, S) \coloneqq \{u \in \SobolevH^1_\loc(M, S) \cap \Lp^2(M, S) \mid \Dirac u \in \Lp^2(M,S), \RelDiracInv \tau(u) = s \tau(u) \}.
\]
We will endow \(\SobolevH^1_{\RelDiracInv, s}(M, S)\) with the norm defined by
\begin{equation}
  \|u\|_{\SobolevH^1_{\RelDiracInv, s}(M, S)}^2 \coloneqq \|u\|^2 + \|\Dirac u\|^2. \label{eq:Dirac-graph-norm}
\end{equation}
It also follows from ellipticity (also of the boundary condition) that the inclusion \(\SobolevH^1_{\RelDiracInv, s}(M, S) \hookrightarrow \SobolevH^1_\loc(M, S)\) is continuous.
In particular, the trace operator is continuous on \(\SobolevH^1_{\RelDiracInv, s}(M, S)\).
The next two lemmas allow to also describe the domain of the closure of \(\Callias_{\psi, s}\), temporarily denoted by \(\bar\Callias_{\psi,s}\), and show that compact perturbations of the potential do not alter any relevant properties.

\begin{lem}\label{domain-does-not-change-lemma}
  Let $\psi$ be an admissible potential.
  Then \(\bar\Callias_{\psi,s}\) is self-adjoint and $\dom\bigl(\bar\Callias_{\psi,s}\bigr)=\SobolevH^1_{\RelDiracInv, s}(M, S)$.
  Moreover, the graph norm induced by $\bar\Callias_{\psi,s}$ is equivalent to the norm given by \labelcref{eq:Dirac-graph-norm}.
\end{lem}
\begin{proof}
This is a direct consequence of the fact that the difference $\Callias_{\psi}-\Dirac=\psi\RelDiracInv$ is in $\Lp^\infty(M,\End(S))$ and (fiberwise) self-adjoint, and thus extends to a self-adjoint bounded operator on $\Lp^2(M,S)$.
\end{proof}

From now on, we will drop the notational distinction between \(\Callias_{\psi,s}\) and its closure \(\bar \Callias_{\psi,s}\), and simply view \(\Callias_{\psi,s}\) to be defined on \(\SobolevH^1_{\RelDiracInv, s}(M, S)\) unless specified otherwise.

\begin{lem}\label{L:Lipschitz->smooth}
  Let $\psi_1$ and $\psi_2$ be admissible potentials coinciding outside a compact set $L$.
  Then $\Callias_{\psi_2}-\Callias_{\psi_1}$ defines a compact operator $\SobolevH^1_{\RelDiracInv, s}(M, S)\to \Lp^2(M,S)$.
  In particular, viewed as maps $\SobolevH^1_{\RelDiracInv, s}(M, S) \to \Lp^2(M, S)$, the operator $\Callias_{\psi_1,s}$ is Fredholm if and only if $\Callias_{\psi_2,s}$ is.
\end{lem}

\begin{proof}
Observe that, that, since $\psi_1$ and $\psi_2$ are Lipschitz and coincide in $M\setminus L$, both \(\eta \coloneqq \psi_2 - \psi_1\) and its differential $\dd \eta $ are essentially bounded and supported in the compact set $L$.
Together with the estimate 
\begin{align*}
 \|\Dirac (\Callias_{\psi_2}-\Callias_{\psi_1})u\| &= \| \Dirac \eta \RelDiracInv u \|\\
 &\leq \|\eta \RelDiracInv \Dirac u\| + \|\clm(\dd \eta) u \| \\
 &\leq  \|\eta\|_{\infty} \|\Dirac u\| + \|\dd \eta\|_{\infty} \|u\|,
\end{align*}
this implies that $\Callias_{\psi_2}-\Callias_{\psi_1}= \eta \RelDiracInv$ defines a bounded operator $\SobolevH^1_{\RelDiracInv, s}(M, S)\to \SobolevH^1_{\RelDiracInv, s}(M, S)$.
Moreover, since \(\eta\) is supported in the compact subset \(L\) and using continuity of the inclusion \(\SobolevH^1_{\RelDiracInv, s}(M, S) \hookrightarrow \SobolevH^1_\loc(M, S)\), this actually means that $\Callias_{\psi_2}-\Callias_{\psi_1}$ is a bounded operator $\SobolevH^1_{\RelDiracInv, s}(M, S)\to \SobolevH^1(L, S)$.
Finally, since the inclusion $\SobolevH^1(L,S)\hookrightarrow \Lp^2(M,S)$ is compact by the Rellich lemma, $\Callias_{\psi_2}-\Callias_{\psi_1}$ yields a compact operator $\SobolevH^1_{\RelDiracInv, s}(M, S)\to \Lp^2(M,S)$.
The second claim follows from classical properties of Fredholm operators.
\end{proof}

\begin{thm} \label{thm:selfadjoint-and-fredholm}
  Let \(M\) be a complete Riemannian manifold with compact boundary \(\partial M\) and \(S \to M\) be a relative Dirac bundle.
  Let \(\psi \colon M \to \mathbb{R}\) be an admissible potential and \(s \colon \partial M \to \{\pm 1\}\) a choice of signs.
  Then the operator \(\Callias_{\psi, s}\) is self-adjoint and Fredholm.
\end{thm}

\begin{proof}
  Since for any given admissible potential, we can always find a smooth admissible potential which agrees with the original one outside a compact subset, \cref{domain-does-not-change-lemma,L:Lipschitz->smooth} imply that we can assume without loss of generality that $\psi$ is smooth.
  As \(\Callias_\psi\) has the same principal symbol as \(\Dirac\), the canonical boundary Dirac operator \(\BdDirac\) from \labelcref{eq:canonical-boundary-operator} is also an adapted operator for \(\Callias_\psi\) in the sense of \cite[Section~3.2]{Baer-Ballmann:Guide-Boundary-value-problems-Dirac}.
  Thus the discussion in the paragraphs preceding \cref{domain-does-not-change-lemma} applies verbatim with \(\Callias_\psi\) replaced by \(\Dirac\).
  Furthermore, since \(\Callias_\psi^2 = \Dirac^2 + \clm(\D \psi) \RelDiracInv + \psi^2\) and \(\psi^2 - |\D \psi|\) is uniformly positive outside a compact subset because \(\psi\) is admissible, the operator \(\Callias_\psi\) is coercive at infinity in the sense of \cite[Definition~5.1]{Baer-Ballmann:Guide-Boundary-value-problems-Dirac}.
  Thus \(\Callias_{\psi, s}\) is a Fredholm operator by \cite[Theorem~5.3]{Baer-Ballmann:Guide-Boundary-value-problems-Dirac}.
\end{proof}

In particular, we obtain an index
\[
    \ind(\Callias_{\psi, s}) \coloneqq \ind(\Callias_{\psi, s}^+) \coloneqq \dim \ker(\Callias_{\psi, s}^+) - \dim \ker (\Callias_{\psi, s}^-) \in \Z.
\]
As another immediate consequence of \cref{L:Lipschitz->smooth}, we obtain the following statement.
\begin{lem} \label{lem:potential-agrees-outside-compact}
  Let \(M\) be a complete Riemannian manifold with compact boundary \(\partial M\) and \(S \to M\) be a relative Dirac bundle.
  Let \(\psi_1, \psi_2 \colon M \to \mathbb{R}\) be two admissible potentials and \(s \colon \partial M \to \{\pm 1\}\) a choice of signs.
  Suppose that \(\psi_1\) and \(\psi_2\) agree outside a compact subset of \(M\).
  Then
  \[
    \ind(\Callias_{\psi_1, s}) = \ind(\Callias_{\psi_2, s}).
  \]
\end{lem}

The next theorem provides the main tool to compute \(\ind(\Callias_{\psi, s})\).

\begin{thm}\label{thm:splitting-theorem}
  Let \(M\) be a complete Riemannian manifold with \(\partial M = \emptyset\) and \(S \to M\) be a relative Dirac bundle of support \(K \subseteq M\).
  Let \(\psi \colon M \to \mathbb{R}\) be an admissible potential.
  Let \(N \subset M \setminus K\) be a compact hypersurface with trivial normal bundle.
  Then we let \(M'\) be the manifold obtained from cutting \(M\) open along \(N\) so that \(\partial M' = N_0 \sqcup N_1\), where \(N_0\) and \(N_1\) are two disjoint copies of \(N\).
  Pulling back all data via the quotient map \(M' \to M\) induces a relative Dirac bundle \(S' \to M'\) and an admissible potential \(\psi' \colon M' \to \R\), respectively.
  Let \(s \colon \partial M' \to \{\pm 1\}\) be any choice of signs such that \(s|_{N_0} = s|_{N_1} \colon N \to \{\pm 1\}\).
  Then
  \begin{equation} \label{eq:decomposition-index-formula}
      \ind(\Callias'_{\psi', s}) = \ind(\Callias_{\psi}),
  \end{equation}
  where \(\Callias_\psi = \Dirac + \psi \RelDiracInv\) denotes the Callias operator on \(M\) and \(\Callias'_{\psi'}\) the corresponding operator on \(M'\).
\end{thm}

\begin{proof}
    Using \cref{L:Lipschitz->smooth,lem:potential-agrees-outside-compact}, it is again enough to prove the claim in the case when $\psi$ is smooth.
    We will prove this case as a consequence of the general splitting theorem due to B\"ar and Ballmann; see \cite[Theorem~6.5]{Baer-Ballmann:Guide-Boundary-value-problems-Dirac} and \cite[Theorem~8.17]{Baer-Ballmann:Boundary-value-problems-first-order}.
    
    First we note that, by the proof of \cref{thm:selfadjoint-and-fredholm} and the remarks preceding it, the boundary condition \(B_\chi^+ \subseteq \SobolevH^{1/2}(\partial M', S'^+)\) defined by the chirality \(\chi = s \clm(\nu^\flat) \RelDiracInv\) restricted to \(S^+\) is elliptic and its adjoint is the corresponding boundary condition \(B_\chi^- \subseteq \SobolevH^{1/2}(\partial M', S'^-)\).
  Thus we can also apply the theory of \citeauthor{Baer-Ballmann:Guide-Boundary-value-problems-Dirac} separately to \(\Callias_\psi^\pm\) and \(\Callias'^\pm_{\psi', s}\).
    In this light, the theorem is a direct consequence of the general splitting theorem~\cite[Theorem~6.5]{Baer-Ballmann:Guide-Boundary-value-problems-Dirac}, \cite[Theorem~8.17]{Baer-Ballmann:Boundary-value-problems-first-order}.
    To see this, observe that \(B_\chi^+ = B_0^+ \oplus B_1^+\) with respect to the decomposition \(\SobolevH^{1/2}(\partial M', S'^+) = \SobolevH^{1/2}(N, S^+) \oplus \SobolevH^{1/2}(N, S^+)\) coming from \(\partial M = N_0 \sqcup N_1\), where
  \[
    B_i^+ = \{ u \in \SobolevH^{1/2}(N, S^+) \mid s \clm(\nu^\flat|_{N_i}) \RelDiracInv u = u \}, \qquad i = 0,1.
  \]
  By construction, the interior normal field of \(\partial M'\) along \(N_0\) is equal to the exterior normal field along \(N_1\) and hence \(\nu|_{N_0} = - \nu|_{N_1}\).
  Thus \((B_0^+)^\perp = B_1^+\) viewed as an \(\Lp^2\)-orthogonal complement in \(\SobolevH^{1/2}(N, S^+)\).
  Consequently, the hypotheses of \cite[Theorem~6.5]{Baer-Ballmann:Guide-Boundary-value-problems-Dirac} are satisfied and we obtain \(\ind(\Callias_{\psi', s}^+) = \ind(\Callias_\psi^+)\).
  This concludes the proof because, by definition, \(\ind(\Callias'_{\psi', s}) = \ind(\Callias'^+_{\psi', s})\) and \(\ind(\Callias_{\psi}) = \ind(\Callias_{\psi}^+)\).
\end{proof}

\begin{lem} \label{lem:uniformly-coercive-potential-vanishing-lemma}
  Let \(M\) be a complete Riemannian manifold with compact boundary \(\partial M\) and \(S \to M\) be a relative Dirac bundle.
  Let \(\psi \colon M \to \mathbb{R}\) be an admissible potential and \(s \colon \partial M \to \{\pm 1\}\) a choice of signs such that 
  \begin{myenumi}
    \item there exists \(C > 0\) such that \(\psi^2 - |\D \psi| \geq C\) on all of \(M\),
    \item \(s \psi \geq 0\) along \(\partial M\).
  \end{myenumi}
  Then \(\Callias_{\psi, s}\) is invertible.
  In particular, \(\ind(\Callias_{\psi, s}) = 0\).
\end{lem}
\begin{proof}
  By Green's formula, for any \(u \in \Css^\infty(M, S)\), we have
  \begin{align*}
    \int_M |\Callias_\psi u|^2\ \vol_M &= \int_{M} \left(|\Dirac u|^2 + \langle u, \clm(\D \psi) \RelDiracInv u + \psi^2 u \rangle\ \right) \vol_M 
    + \int_{\partial M} \langle u, \psi \underbrace{\clm(\nu^\flat) \RelDiracInv u}_{= s u} \rangle\ \vol_{\partial M} \\
    &\geq \int_M (\psi^2 - |\D \psi|) |u|^2\ \vol_M + \int_{\partial M} \underbrace{s \psi}_{\geq 0} |u|^2\ \vol_{\partial M} \geq C \int_{M} |u|^2\ \vol_M.
  \end{align*}
  By continuity, the final estimate holds for all \(u \in \SobolevH^1_{\RelDiracInv, s}(M, S)\).
  Therefore, \(\Callias_{\psi,s}\) is invertible (compare~\cite[Corollary~5.9]{Baer-Ballmann:Guide-Boundary-value-problems-Dirac}).
\end{proof}

\begin{lem} \label{lem:Callias-index-on-compact-mfd}
  Let \(M\) be a compact Riemannian manifold with boundary \(\partial M\) and \(S \to M\) be a relative Dirac bundle with empty support \parensup{that is, the involution \(\RelDiracInv\) is defined on all of \(M\)}.
  Let \(s \colon \partial M \to \{\pm 1\}\) be a choice of signs.
  Then any Lipschitz function \(\psi \colon M \to \R\) is an admissible potential and \(\ind(\Callias_{\psi, s})\) does not depend on the choice of the potential \(\psi\).
  Furthermore, if the sign \(s \in \{\pm 1\}\) is constant on all of \(\partial M\), then \(\ind(\Callias_{\psi, s}) = 0\) for any potential \(\psi\).
\end{lem}
\begin{proof}
  Since \(M\) is compact and \(\RelDiracInv\) is defined on of all \(M\), the condition of being an \emph{admissible} potential is vacuous.
  Moreover, by \cref{lem:potential-agrees-outside-compact}, any two potentials yield the same index.
  Finally, if the choice of signs \(s\) is constant, then \cref{lem:uniformly-coercive-potential-vanishing-lemma} implies vanishing of the index for the constant potential \(\psi = s\).
  Since the index does not depend on the potential, it must vanish for any choice of potential.
\end{proof}

Let us now specialize to the case of \cref{ex:Gromov-Lawson-Relative}.
Let $M$ be a compact even-dimensional Riemannian manifold with boundary.
Let $S$ be the Dirac bundle associated to a GL pair $(E,F)$ over $M$.
For an admissible potential $\psi$, consider the Callias operator $\Callias_\psi$.
With the choice of sign $s=1$, let us consider the index of $\Callias_{\psi, 1}$.
In order to compute \(\ind(\Callias_{\psi,1})\), we make use of the following construction.
  Form the double $\double M \coloneqq M\cup_{\partial M} M^-$ of $M$, where $M^-$ denotes the manifold $M$ with opposite orientation.
  Observe that $\double M$ is a closed manifold equipped with a spin structure induced by the spin structure of $M$.
  Using Condition~\eqref{GromovLawsonAssumption}, let $V(E,F)\to \double M$ be a bundle on $\double M$ which outside a small collar neighborhood coincides with $E$ over $M$ and with $F$ over $M^-$ defined using the bundle isomorphism implicit in a GL pair.
  The \emph{relative index} of $(E,F)$ is the index of the spin Dirac operator on $\double M$ twisted with the bundle $V(E,F)$, that is,
  \[
  \relind(M;E,F)\coloneqq\ind\left(\ReducedSpinDirac_{\double M,V(E,F)}\right)\in\Z.
  \]
  The computation of \(\ind(\Callias_{\psi,1})\) is given by the next proposition.

\begin{cor}\label{relative-index-computation}
Consider the setup of \cref{ex:Gromov-Lawson-Relative}.
Then for any choice of potential \(\psi\), we have
\[
    \ind(\Callias_{\psi,1})=\relind(M;E,F),
\]
where the latter expression is described in the paragraph preceding this corollary.
\end{cor}

\begin{proof}
Let \(E^\prime = V(E,F)\).
Moreover, extend the bundle $F$ to a bundle with metric connection $F^\prime$ on $\double M$ such that $F^\prime|_{M^-}=F$.
Consider the $\Z/2$-graded Dirac bundle \(W^\prime \coloneqq \slashed S_{\double M} \tensgr (E^\prime \oplus (F^\prime)^{\op})\) with associated Dirac operator $\Dirac^\prime$.
Observe that
\begin{equation}\label{GLIndex}
    \ind\Dirac^\prime=\relind(M;E,F)
\end{equation}
because the index of the Dirac operator on \(\slashed S_{\double M} \otimes F^\prime\) vanishes.
Observe also that $\Dirac^\prime=\Callias_0^{E^\prime,F^\prime}$.
Cut $\double M$ open along $\partial M$ as in \cref{thm:splitting-theorem}.
By pulling back all data, we obtain the operators $\Callias_{0,1}^{E,F}$ on $M$ and $\Callias_{0,1}^{F,F}$ on $M^-$.
By~\cref{lem:Callias-index-on-compact-mfd}, \(\ind \Callias_{0,1}^{F,F}=0\).
By~\cref{lem:potential-agrees-outside-compact}, \(\ind (\Callias_{\psi,1}^{E,F})=\ind(\Callias_{0, 1}^{E,F})\).
Therefore, the thesis follows using Identity~\eqref{GLIndex} and~\cref{thm:splitting-theorem}.
\end{proof}    

We now deal with \cref{ex:SimpleBand} in a similar fashion.
\begin{cor}
\label{ex:band-index-conputation}
  Consider the setup of \cref{ex:SimpleBand} and choose the signs such that \(s|_{\partial_\pm V} = \pm 1\).
  Then for any choice of potential \(\psi\), we have
  \[
    \ind(\Callias_{\psi, s}) = \ind(\ReducedSpinDirac_{\partial_- V, E|_{\partial_- V}}) = \int_{\partial_- V} \AhatClass(\partial_- V) \wedge \ch(E)|_{\partial_- V},  
  \]
  where \(\ReducedSpinDirac_{\partial_- V, E|_{\partial_- V}}\) denotes the corresponding twisted spinor Dirac operator on \(\partial_- V\), \(\AhatClass(\partial_- V)\) is the \(\Ahat\)-form of \(\partial_- V\), \(\ch(E)\) the Chern character form associated to \(E\).
  \end{cor}
  \begin{proof}
  First of all, the index does not depend on \(\psi\) by \cref{lem:Callias-index-on-compact-mfd} since \(V\) is compact.
  We can thus choose a function \(\psi\) suitable for our purposes.
  Furthermore, let
  \[
    V_\infty = V_- \cup_{\partial_- V} V \cup_{\partial_+ V} V_+,
    \qquad \text{where}\ V_- \coloneqq \partial_- V \times (-\infty, -1],
    \quad V_+ \coloneqq \partial_+ V \times [1, \infty),
  \]
  be the infinite band obtained from attaching infinite cylinders along the boundary parts \(\partial_\pm V\).
  We extend the Riemannian metric on \(V\) to a complete metric on \(V_\infty\).
  Then the same construction as in \cref{ex:SimpleBand} yields a relative Dirac bundle on \(V_\infty\) extending the data on \(V\).
  Now choose a smooth function \(\psi_\infty \colon V_\infty \to [-1,1]\) such that \(\psi_\infty(V_\pm) = \pm 1\).
  We will denote \(\psi \coloneqq \psi_\infty|_{V}\).
  Applying  \cref{thm:splitting-theorem} to the splitting of \(V_\infty\) along \(\partial V = \partial_- V \sqcup \partial_+ V\) implies that
  \[
    \ind(\Callias_{V_\infty, \psi_\infty}) = \ind(\Callias_{V, \psi, s}) + \ind(\Callias_{V_-, -1, -1}) + \ind(\Callias_{V_+, +1, +1})
  \]
  However, \cref{lem:uniformly-coercive-potential-vanishing-lemma} implies that \( \ind(\Callias_{V_\pm, \pm 1, \pm 1}) = 0\), and so \(\ind(\Callias_{V_\infty, \psi_\infty}) = \ind(\Callias_{V, \psi, s})\).
  Finally, it follows from~\cite[Corollary~1.9]{Anghel:Callias} that \(\ind(\Callias_{V_\infty, \psi_\infty}) = \ind(\ReducedSpinDirac_{\partial_- V, E|_{\partial_- V}})\); for a more general context compare also the partitioned manifold index theorem from~\cite[Appendix A]{Zeidler:band-width-estimates}.
  The last equality follows from the Atiyah--Singer index theorem~\cite[Theorem~13.10]{LawsonMichelsohn:SpinGeometry}.
  \end{proof}
\section{Spectral estimates} \label{sec:curvature-estimates}
Our goal here is to establish spectral estimates for the Callias operator \(\Callias_\psi\) from \cref{eq:defi_callias} associated to a relative Dirac bundle \(S \to M\) and an admissible potential \(\psi \colon M \to \R\).
The parts only concerning the Dirac operator are similar to estimates of imaginary eigenvalues in the context of the \enquote{MIT bag boundary conditions} due to \citeauthor{Raulot:Optimal-Eigenvalue-Estimate}~\cite{Raulot:Optimal-Eigenvalue-Estimate}.

We start with a lemma computing the \(\Lp^2\)-norm of the Dirac operator applied to a section in terms of the Bochner--Lichnerowicz--Weitzenböck curvature term and a boundary term.
In the following, we make extensive use of the notation and formulas introduced in \cref{sec:relative Dirac bundles}.

\begin{lem}\label{lem:Dirac-Penrose-Mean-curvature}
  Let \(S \to M\) be a relative Dirac bundle and \(s \colon \partial M \to \{\pm 1\}\) a choice of signs.
  Then for every \(u \in \Ct^\infty_\cpt(M, S)\), the following identity holds.
  \[
    \int_{M} |\Dirac u|^2\ \vol_M = \frac{n}{n-1} \int_M \left(|\Penrose u|^2 + \langle u, \Curv u \rangle \right)\ \vol_M + \int_{\partial M} \langle u, (\tfrac{n}{2} \mean_g - \tfrac{n}{n-1} \BdDirac)u \rangle\ \vol_{\partial M}.
  \]
\end{lem}
\begin{proof}
  Using the Green- and Weitzenböck formulas, we obtain
  \begin{align*}
    \int_{M} |\Dirac u|^2\ \vol_M &= \int_M \langle  u, \Dirac^2 u \rangle\ \vol_M   + \int_{\partial M} \langle u, \nu^\flat \cdot \Dirac u \rangle\ \vol_{\partial M} \\
    &= \int_{M} |\nabla u|^2\ \vol_M + \int_M \langle u, \Curv u \rangle\ \vol_M  + \int_{\partial M} \langle u, (\underbrace{\clm(\nu^\flat) \Dirac + \nabla_\nu}_{=\frac{n-1}{2}\mean_g - \BdDirac})u \rangle\ \vol_{\partial M}.
  \end{align*}
  The identity now follows from \labelcref{eq:boundary-dirac-vs-mean,eq:Penrose-Dirac}.
\end{proof}

We now combine this with another application of Green's formula to get a corresponding computation for the Callias operator.

\begin{prop} \label{thm:callias-estimate}
  Let \(S \to M\) be a relative Dirac bundle and \(\psi \colon M \to \R\) be an admissible potential.
  Then for every \(u \in \Ct^\infty_\cpt(M, S)\), the following identity holds.
  \begin{align}
    \int_M | \Callias_\psi u |^2 =& \frac{n}{n-1} \int_M \left(|\Penrose u|^2 + \langle u, \Curv u \rangle \right)\ \vol_M + \int_M \langle u, (\psi^2 + \clm(\D \psi) \RelDiracInv) u \rangle\ \vol_M \nonumber
    \\ &+ \int_{\partial M} \langle u, (\tfrac{n}{2} \mean_g - \tfrac{n}{n-1} \BdDirac + \psi \clm(\nu^\flat) \RelDiracInv )u \rangle\ \vol_{\partial M}. \label{eq:general-callias-equality}
  \end{align}
\end{prop}
\begin{proof}
  We have
  \[
    | \Callias_\psi u |^2 = |\Dirac u|^2 + \langle \Dirac u, \psi \RelDiracInv u \rangle + \langle \psi \RelDiracInv u, \Dirac u \rangle  + \psi^2 |u|^2 
  \]
  Using Green's formula~\eqref{eq:Dirac-Green} on the second term implies
  \begin{align*}
    \int_M | \Callias_\psi u |^2\ \vol_M =& \int_M |\Dirac u|^2\ \vol_M \\
    &+ \int_M (\langle u, \underbrace{(\Dirac \psi \RelDiracInv + \psi \RelDiracInv \Dirac)}_{= \clm(\D \psi) \RelDiracInv} u \rangle + \psi^2 |u|^2 )\ \vol_M 
    + \int_{\partial M} \langle u, \nu^\flat \cdot  \psi \RelDiracInv u \rangle \ \vol_{\partial M}.
  \end{align*}
  Combining this with \cref{lem:Dirac-Penrose-Mean-curvature} yields the desired identity \eqref{eq:general-callias-equality}.
\end{proof}

Finally, this leads to the main result of this section.
\begin{thm} \label{thm:callias-estimate-with-boundary}
  Let \(S \to M\) be a relative Dirac bundle over a compact manifold \(M\), \(\psi \colon M \to \R\) be an admissible potential and \(s \colon \partial M \to \{\pm 1\}\) be a choice of signs.
  Then for every \(u \in \SobolevH^1_{\RelDiracInv, s}(M, S)\), the following estimate holds.
  \begin{align}
  \int_M | \Callias_\psi u |^2\ \vol_M &\geq \frac{n}{n-1} \int_M \langle u, \mathcal{R} u \rangle\ \vol_M + \int_M \langle u, (\psi^2 + \clm(\D \psi) \RelDiracInv)u \rangle\ \vol_M \nonumber\\
  &\quad + \int_{\partial M} \left(\tfrac{n}{2} \mean_g + s \psi \right) |\tau(u) |^2\ \vol_{\partial M}.\nonumber\\
  &\geq \frac{n}{n-1} \int_M \langle u, \mathcal{R} u \rangle\ \vol_M + \int_M \left(\psi^2 - |\D \psi|\right)|u|^2\ \vol_M \nonumber\\
  &\quad + \int_{\partial M} \left(\tfrac{n}{2} \mean_g + s \psi \right) |\tau(u) |^2\ \vol_{\partial M}.\label{eq:callias-estimate-boundary-condition}
  \end{align}
  Moreover, equality throughout both estimates in \labelcref{eq:callias-estimate-boundary-condition} holds if and only if
  \begin{equation}
    \forall \xi \in \Ct^\infty(M, \T M) \colon \quad \Penrose_\xi u = \nabla_\xi u + \frac{1}{n} \clm(\xi^\flat) \Dirac u = 0 \qquad \text{and} \quad \clm(\D \psi) \RelDiracInv u = -|\D \psi| u. \label{eq:callias-estimate-equality}
  \end{equation}
\end{thm}
\begin{proof}
  We first assume that \(u \in \Css^\infty(M,S)\).
  Then \(\clm(\nu^\flat) \RelDiracInv u|_{\partial M} = s \chi u|_{\partial M} = s u|_{\partial M}\) and \(\langle u|_{\partial M}, \BdDirac u|_{\partial M} \rangle = 0\), see~\labelcref{eq:boundary-involution,eq:BdDirac-vanishes-under-bd-condition}.
  Thus \eqref{eq:general-callias-equality} simplifies to the equality 
  \begin{align}
    \int_M | \Callias_\psi u |^2 =& \frac{n}{n-1} \int_M \left(|\Penrose u|^2 + \langle u, \Curv u \rangle \right)\ \vol_M + \int_M \langle u, (\psi^2 + \clm(\D \psi) \RelDiracInv) u \rangle\ \vol_M \nonumber
    \\ &+ \int_{\partial M} \langle \tau(u), (\tfrac{n}{2} \mean_g + s \psi ) \tau(u) \rangle\ \vol_{\partial M}. \label{eq:boundary-conditions-callias-equality}
  \end{align}
  Now we observe that both sides of the identity \eqref{eq:boundary-conditions-callias-equality} are continuous in \(u\) with respect to the topology of \(\SobolevH^1_{\RelDiracInv, s}(M, S)\).
  Thus \eqref{eq:boundary-conditions-callias-equality} still holds for all \(u \in \SobolevH^1_{\RelDiracInv, s}(M, S)\). 

The first estimate of~\labelcref{eq:callias-estimate-boundary-condition} now follows directly because \(|\Penrose u|^2 \geq 0\).
The second estimate follows from \(\clm(\D \psi) \RelDiracInv \geq - |\D \psi|\).
  Moreover, equality in the first estimate is equivalent to \(\int_M |\Penrose u|^2 = 0\) and thus \(\Penrose u = 0\).
  Equality for the second estimate is equivalent to 
  \[
    \int_{M} \langle u, (\clm(\D \psi)\RelDiracInv + |\D \psi|) u \rangle\ \vol_M = 0.
  \]
  Since the self-adjoint bundle endomorphism \(\clm(\D \psi)\RelDiracInv + |\D \psi|\) is fiberwise non-negative, this is equivalent to and \((\clm(\D \psi)\RelDiracInv + |\D \psi|)u = 0\), as claimed.
\end{proof}

\begin{rem}\label{rem:equality-kernel}
  Suppose that \(u\) lies in the kernel of \(\Callias_\psi\) and at the same time satisfies \labelcref{eq:callias-estimate-equality}.
  Then \(\Dirac u = - \psi \RelDiracInv u\) and so 
  \begin{equation}
    \nabla_\xi u = \frac{\psi}{n} \clm(\xi^\flat) \RelDiracInv u \label{eq:general_connection_formula_extremal_situation}
  \end{equation}
  for all \(\xi \in \T M\).
  On each point where \(\D \psi \neq 0\), this implies
  \[
   \nabla_\xi u = \frac{\psi}{n} \clm(\xi^\flat) \clm\left(\tfrac{\D \psi}{|\D \psi|}\right) u
  \]
  and in particular
  \[
    \nabla_{\frac{\nabla \psi}{|\nabla \psi|}} u = -\frac{\psi}{n}  u.
   \]
\end{rem}

In the final remark of this section, we prepare another technical observation in the context of the extremality case of~\cref{thm:callias-estimate-with-boundary} that we will use on multiple occasions.
Here we assume that the potential has a special form relevant for our applications.

\begin{rem}\label{rem:modified_section}
  Suppose that there is a \(1\)-Lipschitz function \(x \colon M \to I\) for some compact interval \(I \subseteq \R\) and let \(\varphi \colon I \to (0,\infty)\) be a smooth strictly logarithmically concave function such that \(\psi \coloneqq f(x) \coloneqq -n/2\ \varphi'(x)/\varphi(x)\) is an admissible potential for the relative Dirac bundle \(S \to M\) we consider (the geometric meaning of this relationship will become apparent later, compare~\cref{sec:WarpedProductRigidity} below).
  Now suppose that we are in the same situation as in \cref{rem:equality-kernel}, that is, we have an element \(u \in \ker(\Callias_{\psi,s}) \subseteq \SobolevH_{\RelDiracInv, s}^1(M,S)\) that realizes equality in \cref{thm:callias-estimate-with-boundary}.
  Then the section \(w \coloneqq \varphi(x)^{-\frac{1}{2}} u\) still lies in \(\SobolevH^{1}_{\RelDiracInv, s}(M, S)\) because \(\varphi(x)^{-\frac{1}{2}}\) is a Lipschitz function.
  Using the elementary computation
  \[
    \D\left(\varphi(x)^{-\frac{1}{2}}\right)(\xi) = \varphi(x)^{-\frac{1}{2}} \frac{\psi}{n}\ \D x(\xi),
  \]
  we deduce from \labelcref{eq:general_connection_formula_extremal_situation} that the (weak) covariant derivative of \(w\) satisfies almost everywhere
  \begin{equation}
    \nabla_\xi w = \frac{\psi}{n} (\D x(\xi) + \clm(\xi^\flat) \RelDiracInv) w \label{eq:modified-section-cov-derivative_remark-general}
  \end{equation}
  for every smooth vector field \(\xi\) on \(M\).
  In particular, \(\nabla w = 0\) on \(K\).
  Moreover, on the set \(\{\clm(\D x) u = \RelDiracInv u\} \coloneqq \{p \in M \setminus K \mid \clm(\D_p x) u_p = \RelDiracInv u_p\}\) we deduce that
  \begin{align*}
    \nabla_\xi w {=} &\frac{\psi}{n} (\D x(\xi) + \clm(\xi^\flat) \RelDiracInv) w \\
    {=} &\frac{\psi}{n} (\D x(\xi) + \clm(\xi^\flat) \clm(\D x)) w \\
    = &\frac{\psi}{n} \clm(\xi^\flat \wedge \D x ) w,
  \end{align*}
  where in the last step we use the the Clifford relation \(\clm(\D x) \clm(\xi^\flat) + \clm(\xi^\flat) \clm(\D x) = -2 \D x(\xi)\) and the notation \(\clm(\xi^\flat \wedge \D x ) \coloneqq \tfrac{1}{2}(\clm(\xi^\flat) \clm(\D x)  - \clm(\D x) \clm(\xi^\flat))\).
  
  In summary,
  \begin{equation}
  \nabla_\xi w = \frac{\psi}{n} \clm(\xi^\flat \wedge \D x ) w \qquad\text{almost everywhere on \(\{\clm(\D x) u = \RelDiracInv u \}\)}. \label{eq:modified-section-cov-derivative_dx1}
  \end{equation}
  Finally, since \(\clm(\xi^\flat \wedge \D x)\) is an anti-self-adjoint bundle endomorphism, \labelcref{eq:modified-section-cov-derivative_dx1} further implies that \(\D(|w|^2)(\xi) = 2 \langle w, \nabla_\xi w \rangle = 0\) almost everywhere on \(\{\clm(\D x) u = \RelDiracInv u \}\).
  Together with \labelcref{eq:modified-section-cov-derivative_remark-general} this implies that 
  \begin{equation} \label{eq:modified_section_derivative_of_norm_vanishes}
    \D |w|^2 = 0 \qquad \text{almost everywhere on \(K \cup \{\clm(\D x) u = \RelDiracInv u\}\)}.
  \end{equation}
  A priori, this only holds in the weak sense, but it still implies that \(|w|^2\) is a constant function if \(M\) is connected and we have \(K \cup \{\clm(\D x)u = \RelDiracInv u \} = M\).
  
  We also note that since \(\D \psi = f'(x) \D x\) and \(f'(x) > 0\), the second part of \labelcref{eq:callias-estimate-equality} implies that \(\{\clm(\D x) u = \RelDiracInv u\}\) contains the set of points in \(M \setminus K\) where \(|\D x| = 1\).
  But this is not a priori satisfied everywhere just under the hypotheses of this remark, and so we will verify the condition \(\clm(\D x)u = \RelDiracInv u\) directly in the specific applications when needed. 
\end{rem}

\section{A long neck principle with mean curvature}\label{sec:LongNeck}
In this section, we establish our long neck principle for Riemannian spin manifolds with boundary, relating the length of the neck to the scalar curvature in the interior and the mean curvature of the boundary.
Based on the technical preparations from the previous sections, we now directly enter into the proof of \cref{LongNeckProblem} starting with the following lemma.
\begin{lem}\label{lem_long_neck_GL_pair}
    Let \(M\) be a compact spin manifold with boundary and \(\Phi \colon M \to \Sphere^n\) a smooth map that is locally constant near \(\partial M\) and of non-zero degree, where \(n = \dim M \geq 2\) is even.
    Set \(l = \dist_g(\supp(\D \Phi), \partial M) > 0\).
    Then there exists a GL pair \((E,F)\) in the sense of \cref{ex:Gromov-Lawson-Relative} such that
    \begin{myenumi}[series=long_neck_proof_conditions]
    \item \(\Curv^{E \oplus F}_p \geq -a(p) \cdot n(n-1)/4\) at each point \(p \in M\), where \(a(p)\) is the area contraction constant of \(\Phi\) at \(p\); \label{item:EF_curvature_bound}
    \item $(E,F)$ has support $K \coloneqq \{p \in M \mid \dist_g(p, \partial M) \geq l \} \supseteq \supp(\dd \Phi)$; \label{item:EF_supported}
    \item $\relind(M;E,F)\neq 0$. \label{item:rel_ind_nonvanish}
\end{myenumi}
\end{lem}
\begin{proof}
    We fix a base-point \(\ast \in \Sphere^n\).
    Since \(\Phi\) is locally constant on \(M \setminus K\) and \(\partial M\) has only finitely many components, there exist finitely many distinct points \(q_1, \dotsc, q_k \in \Sphere^n\) such that \(\Phi(M \setminus K) = \{q_1, \dotsc, q_k\}\).
    Let \(\Omega_i = \Phi^{-1}(q_i) \cap M \setminus K\).
    Then each \(\Omega_i\) is an open subset and \(M \setminus K = \bigsqcup_{i=1}^k \Omega_i\).
    By continuity, \(\Phi(\overline{\Omega_i}) = \{q_i\}\) and hence the closures \(\overline{\Omega_i}\), \(i = 1, \dotsc, k\), form a family of pairwise disjoint closed subsets of \(M\).
    Thus there exist open neighborhoods \(V_i \supset \overline{\Omega_i}\) that are still pairwise disjoint together with smooth functions \(\nu_i \colon M \to [0,1]\) such that \(\nu_i = 1\) on \(\overline{\Omega}_i\) and \(\nu_i = 0\) on \(M \setminus V_i\).
    Next we choose geodesic curves \(\gamma_i \colon [0,1] \to \Sphere^n\) such that \(\gamma_i(0) = \ast\) is the base-point and \(\gamma_i(1) = q_i\).
    From this we obtain a smooth map (see \cref{fig:Psi_construction})
    \[
        \Psi \colon M \to \Sphere^n, \qquad \Psi(p) = \begin{cases}
            \gamma_i(\nu_i(p)) & \text{if \(p \in V_i\),} \\
            \ast & \text{if \(p \in M \setminus \bigsqcup_{i=1}^k V_i\).}
        \end{cases}
    \]
    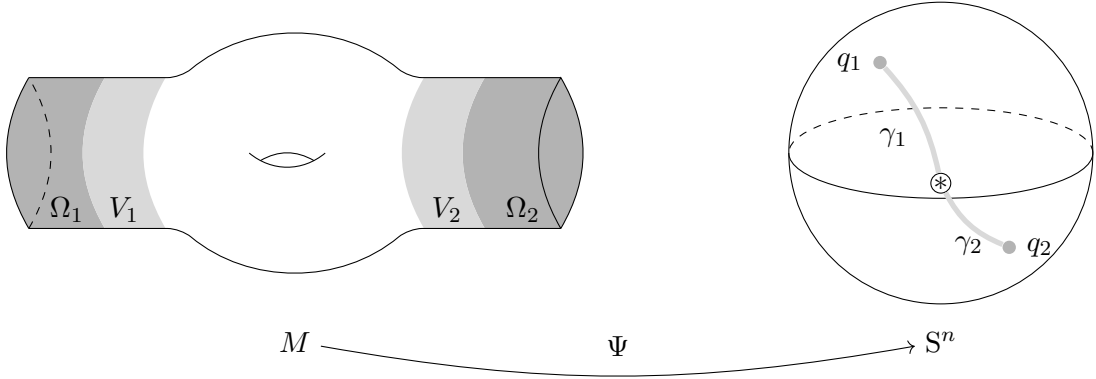
\begin{figure}[h]
\begin{center}
\begin{tikzpicture}
\coordinate (A) at (0,-1);
\coordinate (B) at (0,1);
\coordinate (C) at (-7,1);
\coordinate (D) at (-7,-1);
\coordinate (M) at (5,0);

\fill [gray!60] (A) to [bend left=30] (B) to ++ (-1,0) to [bend right=30] ($(A) + (-1,0)$) -- cycle;
\fill [gray!30] ($(A) + (-1,0)$) to [bend left=30] ($(B) + (-1,0)$) to ++ (-0.8,0) to [bend right=30] ($(A) + (-1.8,0)$) -- cycle;

\fill [gray!60] (D) to [bend left=30] (C) to ++ (1,0) to [bend right=30] ($(D) + (1,0)$) -- cycle;
\fill [gray!30] ($(D) + (1,0)$) to [bend left=30] ($(C) + (1,0)$) to ++ (0.8,0) to [bend right=30] ($(D) + (1.8,0)$) -- cycle;

\node at ($(A) + (-0.5,0.25)$) {\(\Omega_2\)} ;
\node at ($(A) + (-1.5,0.25)$) {\(V_2\)} ;
\node at ($(D) + (0.5,0.25)$) {\(\Omega_1\)} ;
\node at ($(D) + (1.25,0.25)$) {\(V_1\)} ;

\draw[fill=gray!60] (A) to [bend left=30] (B) to [bend left=30] (A);

\draw (D) to [bend left=30] (C);
\draw [dashed] (D) to [bend right=30] (C);

\draw [rounded corners=5] (A) to ++(-2,0) to [bend left=40]  ($ (D) + (2,0) $) to (D);
\draw [rounded corners=5] (B) to ++(-2,0) to [bend right=40]($ (C) + (2,0) $) to  (C);

\draw[bend left=40] (-3.1,-0) edge[name path=hd] (-4.1,-0);
\path[name path=hb] (-3.1,-0.2) to [bend right=40] (-4.1,-0.2);
\path[name intersections={of=hd and hb}];
\draw[bend right=30] (intersection-1) to (intersection-2);

\node (Mlabel) at (-3.5,-2.5) {\(M\)};

\draw (M) circle (2);
\draw ($(M) + (-2,0)$) arc (180:360:2 and 0.6);
\draw[dashed] ($(M) + (2,0)$) arc (0:180:2 and 0.6);

\node[circle,draw=black,minimum size=4pt,inner sep=0pt] (base) at ($(M) + (0,-0.4)$) {\(\ast\)};
\node[circle,fill=gray!60,minimum size=5pt,inner sep=0pt,label=left:{\(q_1\)}] (q1) at ($(M) + (-0.8,1.2)$) {};
\node[circle,fill=gray!60,minimum size=5pt,inner sep=0pt,label=right:{\(q_2\)}] (q2) at ($(M) + (0.9,-1.25)$) {};

\draw[line width=2pt,draw=gray!30] (base) to [bend right=15] node[below left] {\(\gamma_1\)} (q1);
\draw[line width=2pt,draw=gray!30] (base) to [bend right=20] node[below] {\(\gamma_2\)} (q2);

\node (Spherelabel) at ($(M) + (0,-2.5)$) {\(\Sphere^n\)};

\draw[->] (Mlabel) to [bend right=10] node[label=above:{\(\Psi\)}] {} (Spherelabel);

\end{tikzpicture}
\caption{Construction of the map \(\Psi\)}\label{fig:Psi_construction}
\end{center}
\end{figure}
    It follows that \(\Psi = q_i = \Phi\) on each \(\overline{\Omega_i}\) and, since \(\Psi\) by definition locally factors through smooth curves, the induced map \(\Psi_\ast \colon \Lambda^2 \T M \to \Lambda^2 \T \Sphere^n\) vanishes.
    
    Llarull's argument~\cite{LLarull} shows the existence of a Hermitian vector bundle \(E_0 \to \Sphere^n\) such that for any smooth map \(\Theta \colon X \to \Sphere^n\) with \(X\) any spin manifold, we have \(\Curv^{\Theta^\ast E_0} \geq -a_\Theta \cdot n(n-1)/4\), where \(a_{\Theta}\) is the area contraction function of \(\Theta\).
   Moreover, if \(X\) is closed and \(\Theta\) has non-zero degree, then \(\ind(\ReducedSpinDirac_{X, \Theta^\ast E_0}) \neq 0\).
   We now define Hermitian bundles \(E = \Phi^\ast E_0\) and \(F = \Psi^\ast E_0\) on \(M\).
   Then \(F\) is a flat bundle because \(\Psi\) induces the zero map on \(2\)-vectors.
   Thus Llarull's estimate for \(E = \Phi^\ast E_0\) shows that \labelcref{item:EF_curvature_bound} holds.
    
 To see that \labelcref{item:EF_supported} holds, we choose pairwise disjoint open balls \(U_1, \dotsc, U_k \subset \Sphere^n\) such that \(q_i \in U_i\) together with a unitary trivialization \(
   \mathfrak{t}_0 \colon E_0|_{U} \xrightarrow{\cong} U \times \C^r,
   \)
   where \(U = \bigsqcup_{i=1}^{k} U_i\).
   On the open set \(\mathcal{N} \coloneqq \Phi^{-1}(U) \cap \Psi^{-1}(U)\), this induces a unitary bundle isomorphism
   \[
       \mathfrak{t} \colon E|_{\mathcal{N}} \xrightarrow{\Phi^\ast \mathfrak{t}_0} \mathcal{N} \times \C^r \xrightarrow{(\Psi^\ast \mathfrak{t}_0)^{-1}} F|_{\mathcal{N}}.
    \]
    Note that \(M \setminus K \subseteq \Phi^{-1}(\{q_1, \dotsc, q_k\}) \cap \Psi^{-1}(\{q_1, \dotsc, q_k\}) \subseteq \mathcal{N}\) by construction and thus \(\mathcal{N}\) is an open neighborhood of \(\overline{M \setminus K}\).
    The bundle isomorphism is not parallel on all of \(\mathcal{N}\), but as \(\Phi = \Psi\) is locally constant on \(M \setminus K\), it follows that \(\mathfrak{t}|_{M \setminus K}\) is parallel.
  Thus the condition \labelcref{GromovLawsonAssumption} from \cref{ex:Gromov-Lawson-Relative} is satisfied showing that \((E,F)\) is a GL pair with support \(K\).
  
  Finally, let \(\Theta \colon \double M = M \cup_{\partial M} M^{-} \to \Sphere^n\) be the smooth map defined by \(\Theta|_{M} = \Phi\) and \(\Theta|_{M^{-}} = \Psi\).
  Then \(\deg(\Theta) = \deg(\Phi) \neq 0\) because \(\Psi\) has zero degree.
  By definition, we have \(\Theta^\ast E_0 = V(E,F)\) and thus
  \[
    \relind(M;E,F) = \ind\left(\ReducedSpinDirac_{\double M, V(E,F)}\right) = \ind\left(\ReducedSpinDirac_{\double M, \Theta^\ast E_0}\right) \neq 0
  \]
  by Llarull's argument.
  This shows that \labelcref{item:rel_ind_nonvanish} holds and concludes the proof of the lemma.
\end{proof}

\begin{proof}[Proof of \cref{LongNeckProblem}]
Suppose, by contradiction, that $\dist_g\bigl(\supp(\dd \Phi),\partial M\bigr) \geq l$ and \(\deg(\Phi) \neq 0\).
We then pick a GL pair \((E,F)\) satisfying the conditions \labelcref{item:EF_curvature_bound,item:EF_supported,item:rel_ind_nonvanish} from \cref{lem_long_neck_GL_pair}.
Moreover, restating the hypotheses of the theorem, we have the curvature bounds
\begin{myenumi}[resume=long_neck_proof_conditions]
    \item \(\scal_g \geq n(n-1)\) on \(M\); \label{item:Long_Neck:scal}
    \item $\mean_g\geq-\tan(nl/2)$ on \(\partial M\). \label{item:Long_neck:mean}
\end{myenumi}
We then consider the relative Dirac bundle \(S \to M\) constructed out of the GL pair \((E,F)\) as in \cref{ex:Gromov-Lawson-Relative}.
Recall from~\labelcref{E:GL-Lichnerowicz} that the relevant curvature term here takes the form
\[
    \Curv=\frac{1}{4}\scal_g+\Curv^{E\oplus F}.
\]
To construct a suitable admissible potential, we first define the \(1\)-Lipschitz function
\[
    x \colon M \to [0,l], \quad x(p) \coloneqq \min(\dist_g(K, p), l).
\]
Since $\dist_g\bigl(K,\partial M\bigr) \geq l$, we have that \(x|_{\partial M} = l\).
Let \(\varphi, f \colon [0, \pi/n) \to \R\) given by \(\varphi(t) = \cos(nt/2)^{2/n}\) and \(f(t) = -n/2\ \varphi'(t)/\varphi(t) = n/2 \tan(n t / 2)\), compare~\cref{rem:optimality}.
 Then we set \(\psi \coloneqq f(x) \coloneqq f \circ x\).
Since \(x\) is \(1\)-Lipschitz, it follows that
\begin{equation}
    \psi^2 - |\D \psi| = f(x)^2 - f'(x) |\D x| \geq f(x)^2 - f'(x) = -\frac{n^2}{4} \label{eq:long_neck_potential_estimate}
\end{equation}
almost everywhere on \(M\).
Moreover, by \labelcref{item:Long_neck:mean},
\begin{equation}
\psi|_{\partial M} = \frac{n}{2}\tan\left(\frac{n l}{2}\right) \geq - \frac{n}{2}\mean_g. \label{eq:long_neck_boundary_potential_estimate}
\end{equation}
Then \(\psi\) is an admissible potential for the relative Dirac bundle \(S\).
\Cref{relative-index-computation} together with \labelcref{item:rel_ind_nonvanish} implies that the corresponding Callias operator subject to the sign \(s = 1\) satisfies
\[
  \ind(\Callias_{\psi, 1}) = \relind(M;E,F)\neq 0.
\]
In particular, there exists an element \(0 \neq u \in \ker(\Callias_{\psi, 1})\).
To analyze \(u\) further, we let \(U \coloneqq \{p \in M \mid \dd_p \Phi \neq 0\}\) and \(U' \coloneqq \{p \in U \mid a(p) < 1/2\}\).
Then, by definition, \(U' \subseteq U\) are open subsets in the interior of \(M\) and \(U  \subseteq K\).
Since \(\Phi\) has non-zero degree, the set \(U\) must be non-empty.
Furthermore, as the map \(\Phi\) is locally constant near the boundary and \(M\) is connected, the intermediate value theorem implies that \(U'\) is also non-empty.
Then using our main spectral estimates~\labelcref{eq:callias-estimate-boundary-condition} from \cref{thm:callias-estimate-with-boundary}, we obtain
\begin{align*}
0 &\geq \frac{n}{n-1} \int_M  \frac{\scal_g}{4}|u|^2 + \langle u, \mathcal{R}^{E \oplus F} u \rangle\ \vol_M + \int_M \langle u, (\psi^2 + \clm(\D \psi) \sigma) u\rangle\ \vol_M \\
  &\qquad + \int_{\partial M} \underbrace{\left(\tfrac{n}{2} \mean_g + \psi \right)}_{\text{\(\geq 0\) by \labelcref{eq:long_neck_boundary_potential_estimate}}} |\tau(u) |^2\ \vol_{\partial M} 
  \intertext{and continuing the estimate using that \(\psi = 0\) on \(U\) and \(\Curv^{E\oplus F} = 0\) on \(M \setminus U\) leads to}
  &\geq \frac{n}{n-1} \int_{U}  \underbrace{\frac{\scal_g}{4}|u|^2 + \langle u, \mathcal{R}^{E \oplus F} u \rangle}_{\text{\(\geq 0\) by \labelcref{item:Long_Neck:scal,item:EF_curvature_bound} }}\ \vol_M \\
  &\quad +  \int_{M \setminus U}  \underbrace{\frac{n}{n-1}\frac{\scal_g}{4}|u|^2 + (f^2(x) - f'(x)) |u|^2}_{\text{\(\geq 0\) by \labelcref{item:Long_Neck:scal} and \labelcref{eq:long_neck_potential_estimate}}} +  f'(x)\langle u, (1 - \clm(\D x) \RelDiracInv)u \rangle\ \vol_M\\
  & \geq \frac{n}{n-1} \int_{U'}  \frac{\scal_g}{4}|u|^2 + \langle u, \mathcal{R}^{E \oplus F} u \rangle\ \vol_M
  + \int_{M \setminus U} f'(x) \langle u, (1 - \clm(\D x) \RelDiracInv)u \rangle\ \vol_M\\
  &\geq \int_{U'} \left(\frac{n^2}{4} - \frac{n^2}{8}\right)|u|^2\ \vol_M  + \int_{M \setminus U} f'(x) \langle u, (1 - \clm(\D x) \RelDiracInv)u \rangle\ \vol_M\geq 0,
  \end{align*}
  where in the last step we used \labelcref{item:Long_Neck:scal} and \labelcref{item:EF_curvature_bound} together with the fact that the area contraction constant is at most \(1/2\) on \(U'\) by definition.
  We conclude that we are in the equality situation of \cref{thm:callias-estimate-with-boundary}.
  Furthermore, since the last two integrands are separately non-negative, we also deduce \(u = 0\) on \(U'\) and \(f'(x) \langle u, (1 - \clm(\D x) \RelDiracInv)u \rangle =0\) on \(M \setminus U\).
  Since \(f'(x) > 0\) and \(|\D x| \leq 1\), the latter implies that \(\clm(\D x) u = \RelDiracInv u\) almost everywhere on \(M \setminus U\). 
  Hence it follows from \labelcref{eq:modified_section_derivative_of_norm_vanishes} in \cref{rem:modified_section} above that the modified section \(w = \varphi^{-\frac{1}{2}} u\) has a constant norm.
  But since \(u\) vanishes on the non-empty open subset \(U'\) and \(\varphi > 0\), this implies that \(u\) vanishes almost everywhere, a contradiction.
\end{proof}
The precise analysis of the equality situation in the proof above is necessary to rule out the case \(\dist_g(\D \Phi, \partial M) = l\).
If we only wanted to establish the non-strict estimate \(\dist_g(\D \Phi, \partial M) \leq l\) for \(\deg(\Phi) \neq 0\), then this could be proved in a simpler way along the same lines as in~\cite[Proof of Theorem~A]{Ce20} by directly showing the operator \(\Callias_{\psi,1}\) must be invertible if \(\dist_g(\D \Phi, \partial M) > l\).

We now show that \cref{LongNeckProblem} is in fact sharp.
The proof of the following proposition is an almost verbatim adaption of a construction due to \citeauthor{GromovLawson:PSCDiracComplete}~\cite[Proposition~6.7]{GromovLawson:PSCDiracComplete}, see also~\cite[Chapter IV, Proposition~6.10]{LawsonMichelsohn:SpinGeometry}.
\begin{prop}\label{prop:long_neck_sharp}
For every \(n \geq 2\), \(\varepsilon > 0\) and $0 < l < \pi/n$, there exists a compact connected \(n\)-dimensional Riemannian spin manifold \((M,g)\) and a smooth area non-increasing map \(\Phi \colon M \to \Sphere^n\) of non-zero degree such that 
\begin{myenumi}
    \item \(\scal_g = n(n-1)\);
    \item $\mean_g = -\tan(nl/2)$;
    \item $\dist_g\bigl(\supp(\dd \Phi),\partial M\bigr) \geq l - \varepsilon$.
\end{myenumi}
In particular, \cref{LongNeckProblem} is sharp.
\end{prop}
\begin{proof}
We again work with the example from~\cref{rem:optimality}, that is, consider \(V \coloneqq \Torus^{n-1} \times [-l,l]\) endowed with the metric \(g_V = \varphi^2 g_{\Torus^{n-1}} + \D x \otimes \D x\), where \(g_{\Torus^{n-1}}\) is the flat torus metric and \(\varphi(t) = \cos(nt/2)^{2/n}\).
Then \(\scal_{g_V} = n(n-1)\) and \(\mean_{g_V} = -\tan(nl/2)\).
Let \(\Sphere^1\) be the circle of radius \(1\) and \(\ast \in \Sphere^1\) a base-point.
We choose a smooth map \(\gamma \colon [-l,l] \to \Sphere^1\) of degree one such that \(\gamma\) takes \([-l,l] \setminus (-\varepsilon, \varepsilon)\) to the base-point \(\ast \in \Sphere^1\).
For any \(\delta > 0\) we can find a finite covering \(\widetilde{\Torus}^{n-1} \to \Torus^{n-1}\) together with a \(\delta\)-Lipschitz map \(h \colon \widetilde{\Torus}^{n-1} \to \Sphere^{n-1}\) of non-zero degree.
Then we set \(M \coloneqq \widetilde{\Torus}^{n-1} \times [-l,l]\) endowed with the lifted metric \(g \coloneqq \varphi^2 g_{\widetilde{\Torus}^{n-1}} + \D x \otimes \D x\).
This clearly still satisfies \(\scal_g = n(n-1)\) and \(\mean_{g} = -\tan(nl/2)\).
Let \(\Theta \colon \Sphere^{n-1} \times \Sphere^1 \to \Sphere^{n}\) be a smooth map of degree \(1\) which factors through the smash product \(\Sphere^{n-1} \wedge \Sphere^1\).
If \(\delta\) is sufficiently small, then it follows that the composition 
\[
    \Phi \colon M = \widetilde{\Torus}^{n-1} \times [-l,l] \xrightarrow{h \times \gamma} \Sphere^{n-1} \times \Sphere^{1} \xrightarrow{\Theta} \Sphere^{n}.
\]
is \((\delta \cdot C)\)-area-contracting for some constant \(C > 0\) which only depends on \(\varepsilon\) and the Lipschitz constants of \(\gamma\) and \(\Theta\); this can be seen using~\cite[Proposition~6.3]{GromovLawson:PSCDiracComplete}. 
By having chosen \(\delta\) sufficiently small, we can thus arrange that \(f\) is area non-increasing.
Moreover, it follows that \(\deg(\Phi) = \deg(h) \neq 0\).
Finally, the support of \(\D \Phi\) is contained in \(\widetilde{\Torus}^{n-1} \times [-\varepsilon, \varepsilon]\) by construction; hence $\dist_g\bigl(\supp(\dd \Phi),\partial M\bigr) \geq l - \varepsilon$, as claimed.
\end{proof}
Note that unlike the extremal examples for the band and collar width estimates discussed in~\cref{rem:optimality}, the bounds from~\cref{LongNeckProblem} are only approximately realized.
Indeed, it is not possible to do better because---unlike in our other results---\cref{LongNeckProblem} actually rules out the equality situation.
\section{Relative K- and \texorpdfstring{\(\Ahat\)}{A-hat}-area}\label{S:K-area}
In this section, we aim to prove \cref{CollarInfiniteKArea} from the introduction.
We begin with a version of K-area for manifolds with boundary.
The classical case has been introduced by Gromov in~\cite[\S 4]{Gromov:KArea}.
A slightly less general variant of the following was recently studied by~\citeauthor{BH20}~\cite{BH20}.
\begin{defi}
    Let \((M,g)\) be a compact orientable Riemannian manifold with boundary.
    An \emph{admissible pair} of bundles over \(M\) is a pair of Hermitian bundles \(E,F \to M\) endowed with metric connections such that there exists a unitary parallel bundle isomorphism \(E|_{\partial M} \xrightarrow{\cong} F|_{\partial M}\) along \(\partial M\).
    We say that an admissible pair \((E,F)\) has a non-trivial Chern number, if there exists a polynomial \(p(\chernclass_0, \chernclass_1, \dotsc)\) in the Chern forms such that
    \begin{equation} \label{eq:rel-chern-number}
        \int_M p(\chernclass_0(E), \chernclass_1(E), \dotsc) - p(\chernclass_0(F), \chernclass_1(F), \dotsc) \neq 0.
    \end{equation}
    The \emph{relative \(\K\)-area} of \((M, \partial M)\) is the infimum of the numbers 
    \[
        \|\FullCurv^{E \oplus F}\|_\infty^{-1},    
    \]
    where \((E,F)\) ranges over all admissible pairs of bundles which have a non-trivial Chern number.
\end{defi}

\begin{rem}
    Since the bundles with all structures are isomorphic along the boundary, the Chern--Weil forms associated to \(E\) and \(F\) agree if pulled-back to the boundary, and so \labelcref{eq:rel-chern-number} is a well-defined cohomological expression.
\end{rem}

\begin{rem}
    As in the classical case, the property of having \emph{infinite} relative K-area does not depend on the Riemannian metric.
\end{rem}

We now relate the notion of infinite relative $\K$-area for a manifold with boundary $M$ with the notion of infinite K-area for its double $\double M$.
For a pair of Hermitian bundles $E$, $F\to M$ which are suitably identified in a neighborhood of $\partial M$, as in \cref{sec:index-theory} we denote by $V(E,F)$ the Hermitian bundle on $\double M$ obtained by glueing $E$ and $F$ over a neighborhood of $\partial M$.

\begin{prop}\label{k-area-relative-vs-double}
    Let \(M\) be a compact manifold with boundary.
    The following conditions are equivalent:
    \begin{myenuma}
        \item The double \(\double M\) has infinite \(\K\)-area in the classical sense. \label{item:infinite-absolute-K}
        \item \((M, \partial M)\) has infinite relative \(\K\)-area. \label{item:infinite-relative-K}
    \end{myenuma}
\end{prop}
\begin{proof}
    We fix a smooth Riemannian metric on \(\double M\) and endow \(M\) with the restricted metric for the purposes of this argument.
    Then the statement would become immediate, if---in the definition of K-area---we only considered bundles which restricted to a fixed tubular neighborhood \(\partial M \times [-1,1] \cong U \subset \double M\) are of \emph{product structure}, that is, they are of the form \(\proj_1^\ast E_\partial\) for some bundle \(E_\partial \to \partial M\) together with the pullback connection. 
    Indeed, a bundle \(E \to \double M\) of this form is just a pair of two bundles on \(M\) which are identified and of product structure on the corresponding collar neighborhood \(U \cap M\) of \(\partial M\).
    Conversely, suppose \(E,F\to M\) have product structure and are identified on the collar neighborhood \(U \cap M\).
    Then the Chern numbers of \(V(E,F)\) agree with the corresponding relative Chern numbers \eqref{eq:rel-chern-number}.
    
    The general case can be reduced to this observation by using a smooth map \(\psi \colon \double M \to M\) such that \(\psi\) is smoothly homotopic to the identity map on \(\double M\) through maps which are the identity on \(\partial M\) and preserve each half of \(\double M\), and such that \(\psi|_{U}\) agrees with the projection onto \(\partial M\). 
    In fact, \(\psi\) can be constructed on a slightly larger tubular neighborhood by manipulating the radial coordinate and setting it to be the identity on the rest of \(\double M\).
    Then, for any bundle \(E \to \double M\), the pullback \(\psi^\ast E\) is of the special form described in the previous paragraph.
    Since \(\psi\) is homotopic to the identity, \(\psi^\ast E\) and \(E\) have the same Chern numbers (and the same applies to the relative Chern number \labelcref{eq:rel-chern-number} when passing from a pair of bundles \((E,F)\) to \((\psi|_{M}^\ast E, \psi|_{M}^\ast F)\)).
    Moreover, \(\psi\) is a smooth map on a compact manifold, it admits some fixed Lipschitz constant \(L > 0\), and hence passing from \(E\) to \(\psi^\ast E\) only changes the norm of the curvature by at most a factor of \(L^2\).
    Hence, for the purposes of detecting \emph{infinite} K-area, it makes no difference to restrict to the class of (pairs of) bundles described in the previous paragraph, and the proposition follows.
\end{proof}

The notion of K-area is quite appealing because it is a purely bundle-theoretic property of the manifold and does not rely on spin structures or index theory.
However, to apply it to positive scalar curvature geometry via spin geometry, the non-vanishing property \labelcref{eq:rel-chern-number} has to be translated (by potentially changing the bundle) into a property which can be used in the Atiyah--Singer index theorem for the Dirac operator. 
This is possible by the classical algebraic argument using Adams operations given in \cite[\S\(5 \frac{3}{8}\)]{Gromov:KArea} (see also~\citeauthor{BH20}~\cite[Lemma~7]{BH20} for a situation closer to the present context).
In the following, we introduce an explicit notion to capture the resulting property.

\begin{defi}
    Let \((M,g)\) be a compact Riemannian manifold with boundary.
    The \emph{relative \(\Ahat\)-area} of \((M, \partial M)\) is the infimum of the numbers 
    \[
        \|\FullCurv^{E \oplus F}\|_\infty^{-1},    
    \]
    where \((E,F)\) ranges over all admissible pairs of bundles such that
    \begin{equation}
        \int_M \AhatClass(M) \wedge \ch(E) - \int_M \AhatClass(M) \wedge \ch(F) \neq 0.
        \label{eq:rel-Ahat-number}
    \end{equation}
    Here \(\AhatClass(M)\) denotes the \(\Ahat\)-form of \(M\) and \(\ch\) denotes the Chern character form.
\end{defi}

\begin{rem} \label{rem:Ahat-area-without-boundary}
    The notion of $\Ahat$-area for closed manifolds was discussed in \cref{def:Ahat area}.
    An analogous argument as in \cref{k-area-relative-vs-double} shows that \((M, \partial M)\) has infinite relative \(\Ahat\)-area if and only if the double \(\double M\) has infinite \(\Ahat\)-area.
\end{rem}

In the next proposition, we clarify the relationship between the notions of infinite relative K-area and infinite relative \(\Ahat\)-area.

\begin{prop}
    Let \(M\) be a manifold with boundary of infinite relative \(\K\)-area and \(N\) a closed manifold such that \(\Ahat(N) \neq 0\).
    Then \(M \times N\) has infinite relative \(\Ahat\)-area.
    In particular, infinite relative \(\K\)-area implies infinite relative \(\Ahat\)-area.
\end{prop}
\begin{proof}
    Let \(\varepsilon > 0\).
    We start with an admissible pair of bundles \((E,F)\) on \(M\) such that \(\|\FullCurv^{E \oplus F}\|_\infty < \varepsilon\) and such that \eqref{eq:rel-chern-number} is satisfied.
    Arguing as in the proof of \cref{k-area-relative-vs-double}, we may assume without loss of generality that on a collar neighborhood \(U\) of the boundary the bundles \(E,F\) are of product structure and that there exists a parallel unitary bundle isomorphism \(\Phi \colon E|_{U} \to F|_{U}\).
    Consider the bundle \(V(E,F)\) on \(\double M\).
    It follows that
    \[
        \int_{\double M} p(\chernclass_0(V(E,F)), \chernclass_1(V(E,F)), \dotsc) \neq 0.
    \]
    Now using the fact that \(\Ahat(\double M) = 0\) (as the double is null-bordant), the argument given in \cite[\S\(5 \frac{3}{8}\)]{Gromov:KArea} shows that (after altering the initial choice of \(\varepsilon\), passing to a bundle associated to \(V(E,F)\), and restricting the new bundle on \(\double M\) to obtain a new pair of bundles on \(M\)), we can achieve that 
    \[\int_{\double M} \AhatClass(\double M) \wedge \ch(V(E,F)) \neq 0. \]
    But this just means that \eqref{eq:rel-Ahat-number} is satisfied for \((E,F)\) on \(M\) and already shows that \(M\) has infinite \(\Ahat\)-area.
    Finally, pulling back the bundles from \(M\) to \(M \times N\) via the projection \(\proj_M \colon M \times N \to M\) yields the desired result for \(M \times N\) because
    \begin{align*}
        \int_{\double M \times N} &\AhatClass(\double M \times N) \wedge \ch(V(\proj_{M}^\ast E, \proj_{M}^\ast F))\\
         &= \int_{\double M \times N} \proj_{\double M}^\ast\AhatClass(\double M) \wedge \proj_N^\ast\AhatClass(N) \wedge \proj_{\double M}^\ast\ch( V(E,F)) \\
         &= \left(\int_{\double M} \AhatClass(\double M) \wedge \ch(V(E,F)) \right) \cdot \Ahat(N) \neq 0. \qedhere
    \end{align*}
\end{proof}

We can also relate this property to enlargeability.
For an \(n\)-dimensional Riemannian manifold \(N\) with boundary, we say that \((N, \partial N)\) is \emph{compactly area-enlargeable} if for any \(\varepsilon > 0\), there exists a \emph{finite} covering \(\bar{N} \to N\) and an \(\varepsilon\)-area-contracting map \(\bar{M} \to \Sphere^n\) of non-zero degree which is locally constant outside a compact subset of the interior.

\begin{prop} \label{area-enlargeable-Ahat-degree}
    If \(N\) is an even-dimensional compact manifold with \((N, \partial N)\) compactly area-enlargeable and there exists a smooth map \(\Phi \colon (M, \partial M) \to (N, \partial N)\) of non-zero \(\Ahat\)-degree, then \((M, \partial M)\) has infinite relative \(\Ahat\)-area.
\end{prop}
\begin{proof}
    By passing to doubles and \cref{k-area-relative-vs-double,rem:Ahat-area-without-boundary} it suffices to consider the case that \(M\) and \(N\) are closed manifolds.
    Since \(N\) is compactly area-enlargeable, for any \(\varepsilon >0\), there exists a Hermitian bundle \(E_0 \to N\) such that \(\|\FullCurv^{E_0}\|_\infty < \varepsilon/L^2\), where \(L\) is a Lipschitz constant for \(\Phi\), and such that \(\ch(E_0)\) as a cohomology class is concentrated and non-trivial in the degrees \(0\) and \(\dim N\) (see for instance~\cite[313]{HankeSchick:Enlargeable}).
    Let \(E = \Phi^\ast E_0\) and \(F\) to be the trivial bundle of the same rank.
    Then it follows that \(\|\FullCurv^{E \oplus F}\|_\infty < \varepsilon\) and
    \[
      \int_M \AhatClass(M) \wedge (\ch(E) - \ch(F)) = \AhatDeg(\Phi) \int_N \ch(E_0) \neq 0. \qedhere
    \]
\end{proof}

In the following lemma, we spell out a technical consequence of infinite \(\Ahat\)-area that will be used in our applications.

\begin{lem} \label{Rel-KArea-on-spin-mfd}
    Let \((M,g)\) be an even-dimensional compact spin manifold of infinite relative \(\Ahat\)-area and \(U \cong \partial M \times (-1,0] \subseteq M\) be an open collar neighborhood.
    Then for any sufficiently small \(\varepsilon_1, \varepsilon_2 > 0\), there exists a pair of Hermitian vector bundles \((E,F)\) on \(M\) such that
    \begin{myenumi}
        \item there exists a parallel unitary bundle isomorphism \(\mathfrak{t} \colon E|_{U_{\varepsilon_1}} \to F|_{U_{\varepsilon_1}}\), where \(U_{\varepsilon_1}\) corresponds to \(\partial M \times (-1+\varepsilon_1, 0]\) in the tubular neighborhood;\label{item:identified-on-collar}
        \item \(\|\FullCurv^{E \oplus F}\|_\infty < \varepsilon_2\); \label{item:karea-curvature}
        \item \(\ind(\Callias_{\psi, 1}) \neq 0\), where \(\Callias_{\psi, 1}\) is the operator considered in \cref{relative-index-computation} associated to the relative Dirac bundle constructed from \((E,F)\) as in \cref{ex:Gromov-Lawson-Relative}. \label{item:karea-non-vanishing-callias}
    \end{myenumi}
\end{lem}

\begin{proof}
    Let \((E,F)\) be an admissible pair of bundles given by infinite \(\Ahat\)-area satisfying \eqref{eq:rel-Ahat-number} and \labelcref{item:karea-curvature}.
    Again arguing as in the proof of \cref{k-area-relative-vs-double}, we may assume that \(E,F\) are of product structure on the collar neighborhood and that also \labelcref{item:identified-on-collar} is satisfied.
    We again consider the bundle \(V(E,F) \) on \(\double M\).
    Using \cref{relative-index-computation} and the Atiyah--Singer index theorem, we deduce
    \[
        \ind(\Callias_{\psi, 1}) = \relind(M;E,F) = \ind(\ReducedSpinDirac_{\double M, V(E,F)}) = \int_{\double M} \AhatClass(\double M) \wedge \ch(V(E,F)) \neq 0,
    \]
   where the final non-vanishing statement is due to \eqref{eq:rel-Ahat-number}. 
   Thus \labelcref{item:karea-non-vanishing-callias} is also satisfied and the proof is complete.
\end{proof}

After all these technical preparations, we are now ready to give a proof of \cref{CollarInfiniteKArea} from the introduction.

\begin{proof}[Proof of \cref{CollarInfiniteKArea}]
For all sufficiently small $d>0$ denote by $\mathcal{N}_d$ the open geodesic collar neighborhood of $\partial M$ of width $d$.
Suppose, by contradiction, that $\mathcal{N}_{l^\prime}$ exists for some $l<l^\prime<\pi/(\sqrt\kappa n)$ such that \(\scal_g \geq \kappa n(n-1)\) in $\mathcal{N}_{l^\prime}$.
Fix $\Lambda\in (l,l^\prime)$.
Then $K_\Lambda\coloneqq M\setminus\mathcal{N}_\Lambda$ is a compact manifold with boundary such that \(\scal_g \geq \kappa n(n-1)\) in $M \setminus K_\Lambda$.
For $r\in \left(0,{\sqrt\kappa n}/{2}\right)$, consider the function $Y_r(t)=r\tan(rt)$, with $t$ varying in $\left[0,{\pi}/{(\sqrt\kappa n)}\right)$.
Observe that $Y_r(0)=0$ and
\[
    \frac{ \kappa n^2}{4}+Y_r^2-\abs{Y_r^\prime}=\frac{ \kappa n^2}{4}-r^2>0.
\]
By choosing $r$ close enough to ${\sqrt\kappa n}/{2}$, we can also ensure that
\begin{equation}\label{eq:Y_r(Lambda)}
    Y_r(\Lambda)>\frac{\sqrt\kappa n}{2}\tan\left(\frac{\sqrt\kappa nl}{2}\right).
\end{equation}
Let $\kappa_0\coloneqq\inf_{p\in K_\Lambda}\scal_g(p)>0$.
By \cref{Rel-KArea-on-spin-mfd,relative-index-computation,ex:Gromov-Lawson-Relative}, there exists a GL pair $(E,F)$ and associated relative Dirac bundle \(S \to M\) such that
\begin{myenumi}
    \item \((E,F)\) and thus $S$ have support $K_\Lambda$;\label{item:collar_EF_distance}
    \item \(4\|\Curv^{E \oplus F}\|_\infty<\kappa_0\);\label{item:InteriorEstimate}
    \item \(\frac{n}{n-1}\|\Curv^{E \oplus F}\|_\infty<\frac{ \kappa n^2}{4}-r^2\);\label{item:NeckEstimate}
    \item $\ind(\Callias_{\psi, 1}) = \relind(M;E,F)\neq 0$ for any admissible potential \(\psi\). \label{item:collar_relind}
\end{myenumi}
We now use the function $Y_r$ to construct an admissible potential.
Let $x\colon M \to [0,\Lambda]$ be the distance function from $K_\Lambda$.
Consider the Lipschitz function $\psi\coloneqq Y_r\circ x$ (which is actually smooth in the complement of \(\partial K_{\Lambda}\)).
Then $\psi|_{K_\Lambda}=0$ and $\psi|_{\partial M}=Y_r(\Lambda)$.
By \ref{item:collar_EF_distance} and \eqref{eq:Y_r(Lambda)}, $\psi$ is an admissible potential satisfying
\begin{equation}\label{E:LongNeckPrinciple10}
    \psi|_{\partial M}>\frac{\sqrt\kappa n}{2}\tan\left(\frac{\sqrt\kappa nl}{2}\right).
\end{equation}
Let $\Callias_{\psi,1}$ be the associated Callias operator subject to the boundary condition coming from the sign \(s = 1\) and \(u \in \dom(\Callias_{\psi,1}) = \SobolevH^1_{\RelDiracInv, 1}(M, S)\).
From \eqref{E:LongNeckPrinciple10}, we deduce
\begin{equation*}
    \int_{\partial M} \left(\tfrac{n}{2} \mean_g + s \psi \right) |u |^2\ \vol_{\partial M}\geq 0.
\end{equation*}
Therefore, the estimates~\labelcref{eq:callias-estimate-boundary-condition} in~\cref{thm:callias-estimate-with-boundary} imply
\begin{equation*}
    \int_M | \Callias_\psi u |^2\ \vol_M \geq \int_M \Theta_{\psi,n} | u |^2\ \vol_M,
\end{equation*}
where $\Theta_{\psi,n}$ is the $\Lp^\infty$-function defined by
\[
    \Theta_{\psi,n}\coloneqq\frac{n}{n-1}\left(\frac{\scal_g}{4}-\bigl|\mathcal{R}^{E\oplus F}\bigr|\right)+\psi^2 - |\D \psi|.
\]
By~\ref{item:InteriorEstimate} and since \(\psi\) is constant on \(K_\Lambda\), in the interior $\interior K_\Lambda$ we have 
\[
    \Theta_{\psi,n}=\frac{n}{n-1}\left(\frac{\scal_g}{4}-\bigl|\mathcal{R}^{E\oplus F}\bigr|\right)
    \geq\frac{n}{n-1}\left(\frac{\kappa_0}{4}-\|\Curv^{E \oplus F}\|_\infty\right)>0.
\]
By~\ref{item:NeckEstimate}, \(\scal_g \geq \kappa n(n-1)\) on \(\mathcal{N}_{l'}\) and since $x$ is 1-Lipschitz, in $M \setminus \interior{K_\Lambda}$ we have
\[
    \Theta_{\psi,n}\geq \frac{ \kappa n^2}{4}+\psi^2-|\D\psi|-\frac{n}{n-1}\|\Curv^{E \oplus F}\|_\infty
    \geq\frac{ \kappa n^2}{4}-r^2-\frac{n}{n-1}\|\Curv^{E \oplus F}\|_\infty>0.
\]
Therefore, there exists a constant $c>0$ such that $\|\Callias_\psi u\|\geq c\|u\|$ for all \(u \in \SobolevH^1_{\RelDiracInv, 1}(M,S)\).
It follows that $\ind(\Callias_{\psi,1})=0$, contradicting~\ref{item:collar_relind}.
\end{proof}

\section{Estimates of bands} \label{sec:bands}
In this section, we prove our statements related to \cref{conj:band-width}, that is, estimates of Riemannian bands under lower bounds on scalar- and mean curvature.
We start with reviewing Gromov's notion of a Riemannian band~\cite{Gromov:MetricInequalitiesScalar} and other relevant concepts to formulate our results more conveniently.
\begin{defi}\ 
  \begin{myenumi}
    \item A \emph{band} is a compact manifold \(V\) together with a decomposition \(\partial V = \partial_- V \sqcup \partial_+ V\), where \(\partial_\pm V\) are unions of components.
    \item A map \(\Phi \colon V \to V'\) between bands is called a \emph{band map} if \(\Phi(\partial_\pm V) \subseteq \partial_\pm V'\).
    \item The \emph{width} \(\width(V,g)\) of a Riemannian band \((V,g)\) is the distance between \(\partial_{-}V\) to \(\partial_{+} V\) with respect to \(g\). 
    \item A \emph{width function} for a Riemannian band \((V,g)\) is a \(1\)-Lipschitz function \(x \colon V \to [t_-,t_+]\) for some real numbers \(t_- < t_+\) such that \(\partial_\pm V \subseteq x^{-1}(t_\pm)\).
  \end{myenumi}
\end{defi}

In other words, a width function is a band map \(V \to [t_-, t_+]\) which is \(1\)-Lipschitz.
A width function \(x \colon V \to [t_-, t_+]\) always satisfies \(t_+ - t_- \leq \width(V, g)\). 
We also have the following converse:
\begin{lem}\label{lem:width-function}
  Let \((V,g)\) be a band. 
  Then there exists a width function \(x \colon V \to [t_-, t_+]\) which is smooth near the boundary of \(V\)  such that \(t_+ - t_- = \width(V,g)\).
  Moreover, for every \(t_- < t_+\) satisfying \(t_+ - t_- < \width(V,g)\), there exists a smooth width function \(x \colon V \to [t_-, t_+]\).
\end{lem}
\begin{proof}
  Let \(w = \width(V,g)\).
  Then we obtain the desired Lipschitz width function by setting
  \[
    x \colon V \to [0, w], \quad x(p) \coloneqq
    \begin{cases}
      d(p, \partial_- V) & \text{if \(d(p, \partial_- V) \leq w/2\)},\\
      w - d(p, \partial_+ V) & \text{if \(d(p, \partial_+ V) \leq w/2\)},\\
      w/2 & \text{otherwise.}
    \end{cases}
  \]
  Note that this \(x\) is already smooth in a neighborhood of \(\partial V\).
  Moreover, if \(t_+ - t_- < w\), we can find \(\varepsilon > 0\) such that \((1+\varepsilon)^{-1}w = t_+ - t_-\). 
  We can then approximate \(x\) by a smooth function \(\tilde{x} \colon V \to [0, w]\) which agrees with \(x\) near \(\partial V\) and is \((1+\varepsilon)\)-Lipschitz.
  Then \((1+\varepsilon)^{-1}\tilde{x} \colon V \to [0, t_+ - t_-]\) is a smooth width function.
  Translating to the interval \([t_-, t_+]\) yields the desired result.
\end{proof}

The following notion is an adaption of the ideas from~\cref{S:K-area} to the situation of bands.
\begin{defi}
  A band \(V\) is said to have \emph{infinite vertical \(\Ahat\)-area}, if for every \(\varepsilon > 0\), there exists a Hermitian vector bundle \(E \to V\) such that \(\| \FullCurv^E \|_\infty < \varepsilon\) and such that we have 
  \begin{equation}
    \int_{\partial_- V} \AhatClass(\partial_- V) \wedge \ch(E|_{\partial_- V}) \neq 0.  \label{eq:vertical-ahat-condition}
  \end{equation}
\end{defi}

\begin{ex}
  A simple example of an infinite vertical \(\Ahat\)-area is an \emph{\(\Ahat\)-band}, that is a band, such that \(\Ahat(\partial_- V) \neq 0\).
  Another is a band of the form \(V = M \times [-1,1]\), where \(M\) has infinite \(\Ahat\)-area.
\end{ex}
\begin{ex}
  If \(N\) is a closed manifold that is compactly area-enlargeable and there exists a band map \(V \to N \times [0,1]\) of non-zero \(\Ahat\)-degree, then \(V\) has infinite vertical \(\Ahat\)-area (compare \cref{area-enlargeable-Ahat-degree}).
  In particular, this includes the classes of \emph{overtorical bands} introduced by Gromov in \cite{Gromov:MetricInequalitiesScalar} and the generalization of \emph{\(\Ahat\)-overtorical bands} studied in \cite{ZeidlerWidthLargeness}.
\end{ex}

\begin{thm} \label{band-estimate}
  Let \((V,g)\) be a spin band of infinite vertical \(\Ahat\)-area.
  Suppose that \(\scal_g \geq n(n-1)\) and let \(-\pi/n < t_- < t_+ < \pi/n\) such that the mean curvature of \(\partial V\) satisfies 
  \begin{equation} 
    \mean_g|_{\partial_{\pm} V} \geq \mp \tan\left(\frac{n t_{\pm}}{2}\right). \label{eq:band-mean-curvature-assumption}
  \end{equation}
  Then \(\width(V,g) \leq t_+ - t_-\).
\end{thm}
\begin{proof}
  Assume by contradiction that \(\width(V,g) > t_+ - t_-\).
  Let \(t_0 \coloneqq (t_+ + t_-)/2\) be the midpoint between \(t_-\) and \(t_+\).
  Then we can find \(d > 0\) such that \(\width(V,g) > 2d > t_+ - t_-\) and \(-\pi/n < t_0 - d < t_0 + d < \pi/n\).
  \cref{lem:width-function} implies that there exists there exists a smooth width function \(x \colon V \to [-d, d]\).
  
  Now for \(0 < \lambda \leq 1\), we set
  \[
    f_\lambda \colon [-d,d] \to \R, \quad f_{\lambda}(s) = \lambda \tfrac{n}{2} \tan \left(\tfrac{n}{2}(t_0 + \lambda s ) \right).
  \]
  We now fix a \(\lambda_0 < 1\) such that
  \begin{equation}
    \mp \tfrac{n}{2} \tan\left(\frac{n t_{\pm}}{2}\right) \geq \mp f_{\lambda_0}(\pm d). \label{eq:slighly-tangent}
  \end{equation}
  This is possible because the tangent function is increasing and \(t_0 - d < t_- < t_+ < t_0 + d\) by our choice of \(d\).
  We now choose a Hermitian bundle \(E \to V\) satisfying \labelcref{eq:vertical-ahat-condition} and such that the corresponding Weitzenböck curvature endomorphism satisfies \(\Curv^E \geq -\delta\) for some \(\delta < \tfrac{n^2}{4}(1-\lambda_0^2)\). 
 Now we consider the Callias operator \(\Callias_{\psi, s}\) associated to the relative Dirac bundle from \cref{ex:SimpleBand} (using the twisting bundle \(E\)) and with potential \(\psi \coloneqq f_{\lambda_0} \circ x\) and subject to the boundary conditions coming from the choice of signs \(s(\partial_\pm V) = \pm 1\).
 Now \labelcref{eq:vertical-ahat-condition} and \cref{ex:band-index-conputation} imply that \(\ind(\Callias_{\psi, s}) \neq 0\).
 On the other hand, we have
 \[
  \psi^2 - |\D \psi| = f_{\lambda_0}(x)^2 - f_{\lambda_0}'(x) |\D x| \geq f_{\lambda_0}(x)^2 - f_{\lambda_0}'(x) = -\lambda^2_0 \frac{n^2}{4}.
 \]
 Thus \labelcref{eq:callias-estimate-boundary-condition} together with \(\scal_g \geq n(n-1)\) implies that each \(u\) in the domain of \(\Callias_{\psi, s}\) satisfies
 \begin{align*}
   \int_V | \Callias_{\psi, s} u |^2\ \vol_V \geq &\left(\tfrac{n^2}{4}(1 - \lambda_0^2) - \delta \right) \int_V | u |^2\ \vol_M \\
   &+ \int_{\partial_- V} \left(\tfrac{n}{2} \mean_g - f_{\lambda_0}(-d) \right) |u |^2\ \vol_{\partial_- V}\\
   &+ \int_{\partial_- V} \left(\tfrac{n}{2} \mean_g + f_{\lambda_0}(d) \right) |u |^2\ \vol_{\partial_+ V}.
 \end{align*}
 Since the terms \(\tfrac{n}{2} \mean_g|_{\partial_{\pm} V} \pm f_{\lambda_0}(\pm d)\) at the boundary are non-negative by \eqref{eq:band-mean-curvature-assumption} and \eqref{eq:slighly-tangent}, this implies that 
 \[
  \int_V | \Callias_{\psi, s} u |^2\ \vol_V \geq C \int_{V} |u|^2\ \vol_V,
 \]
 where \(C = \tfrac{n^2}{4}(1 - \lambda_0^2) - \delta > 0\) by the choice of \(\delta >0\).
 This shows that the operator \(\Callias_{\psi, s}\) is invertible, a contradiction to \(\ind(\Callias_{\psi, s}) \neq 0\).
\end{proof}

\begin{cor} \label{cor:symmetric-band-estimate}
  Let \((V,g)\) be a spin band of infinite vertical \(\Ahat\)-area.
  Suppose that \(\scal_g \geq n(n-1)\) and that the mean curvature of \(\partial V\) satisfies 
  \[ \mean_g \geq -\tan\left(\frac{n l}{4}\right)\quad \text{for some \(0 < l < \frac{2\pi}{n}\).} \]
  Then \(\width(V,g) \leq l\).
\end{cor}
\begin{proof}
  This is a consequence of \cref{band-estimate} by setting \(t_\pm \coloneqq \pm l/2\).
\end{proof}

\begin{cor} \label{cor:half-band-estimate}
  Let \((V,g)\) be a spin band of infinite vertical \(\Ahat\)-area.
  Suppose that \(\scal_g \geq n(n-1)\) and that the mean curvature of \(\partial V\) satisfies 
  \[ \mean_g|_{\partial_- V} \geq 0 \quad \text{and} \quad \mean_g|_{\partial_+ V} \geq  -\tan\left(\frac{n l}{2}\right)\quad \text{for some \(0 < l < \frac{\pi}{n}\).} \]
  Then \(\width(V,g) \leq l\).
\end{cor}
\begin{proof}
  This is also a consequence of \cref{band-estimate} by setting \(t_- \coloneqq 0\) and \(t_+ \coloneqq l\).
\end{proof}

\begin{cor} \label{cor:total-band-estimate}
  Let \((V,g)\) be a spin band of infinite vertical \(\Ahat\)-area.
  Suppose that \(\scal_g \geq n(n-1)\).
  Then we always have \(\width(V, g) < \frac{2\pi}{n}\).
  Moreover, if \(\partial_- V\) is mean-convex, then \(\width(V, g) < \frac{\pi}{n}\).
\end{cor}
\begin{proof}
  The first statement follows from \cref{cor:symmetric-band-estimate} and the second statement is a consequence of \cref{cor:half-band-estimate}, in both cases because \(-\tan(\vartheta) \to -\infty\) as \(\vartheta \nearrow \pi/2\) but \(\inf_{p \in \partial V} \mean_g > -\infty\) by compactness. 
\end{proof}

\section{The general warped product rigidity theorem}\label{sec:WarpedProductRigidity}
In this section, we establish a general rigidity theorem which allows to compare metrics on bands with certain warped products.

\begin{setup}\label{setup:WarpingFunctions}
  We consider the following setup.
  Let \(n \in \N\), \(I \subseteq \R\) an interval, \(V\) a band, \(x \colon V \to I\) be a continuous function.
  Furthermore, let \(\varphi \colon I \to (0,\infty)\) and \(\kappa, \nu \colon V \to \R\) be smooth functions.
   We define the auxiliary functions \(h, f \colon I \to \R\) by
   \[
     h(t) \coloneqq \frac{\varphi'(t)}{\varphi(t)}, \quad f(t) \coloneqq - \frac{n}{2} h(t)
   \]
   and suppose that the following two conditions are satisfied.
  \begin{itemize}
    \item  The function \(\varphi\) is strictly logarithmically concave, that is, \(h'(t) < 0\) for all \(t \in I\).
    \item For all \(p \in V\) the following inequality holds.
    \begin{equation}
    \frac{n^2 \kappa(p)}{4}  + \frac{n \nu(p)}{n-1} + f(x(p))^2 - f'(x(p)) \geq 0. \label{eq:warping-balance}
  \end{equation}
\end{itemize}
\end{setup}

\begin{rem}\label{rem:warping-balance-motivation}
  The geometric motivation for \eqref{eq:warping-balance} is the following:
  Suppose that we have a warped product metric \(g = \varphi^2 g_M + \D{x} \otimes \D{x}\) on an \(n\)-dimensional band of the form \(V = M \times [t_-,t_+]\).
  Then, if we choose \(\kappa \colon V \to \R\) and \(\nu \colon V \to \R\) such that \(n(n-1) \kappa = \scal_g\) and \(- \nu = \scal_{\varphi^2 g_M}/4 = \frac{1}{4 \varphi^2} \scal_{g_M}\), we precisely obtain the relation
  \[
    \frac{n^2 \kappa(y, t)}{4}  + \frac{n \nu(y, t)}{n-1} + f(t)^2 - f'(t) = 0 \quad \text{for all \((y,t) \in M \times [t_-, t_+]\),}
  \]
  where \(f\) is defined as in \cref{setup:WarpingFunctions} above.
  In this sense, \eqref{eq:warping-balance} abstractly models the scalar curvature equation for warped products.
\end{rem}

\begin{thm} \label{thm:general-warped-product-rigidity}
  Let \((V,g)\) be an \(n\)-dimensional Riemannian spin band and \(E \to V\) be a Hermitian vector bundle endowed with a metric connection.
  We suppose that we have chosen smooth functions \(x \colon V \to I\), \(\varphi, h, f \colon I \to \R\) and \(\kappa, \nu \colon V \to \R\) satisfying \cref{setup:WarpingFunctions}.
  Furthermore, we assume that the following conditions hold:
  \begin{myenumi}
    \item \(\scal_g \geq n(n-1) \kappa\) and \(\Curv^E \geq \nu\); \label{item:curv-bounds}
    \item \(x \colon V \to [t_-, t_+]\) is a width function for some \(t_-, t_+ \in I\); \label{item:width-function}
    \item \(\mean_g|_{\partial_\pm V} \geq \pm h(t_\pm)\), where \(\mean_g\) denotes the mean curvature of \(\partial V\); \label{item:mean-bound}
    \item \(\ker(\Callias_{\psi, s}) \neq 0\), where \(\Callias_{\psi, s}\) is the Callias operator associated to the relative Dirac bundle from \cref{ex:SimpleBand} with potential \(\psi \coloneqq f \circ x\) and subject to the boundary conditions coming from the choice of signs \(s(\partial_\pm V) = \pm 1\). \label{item:callias-kernel}
  \end{myenumi}
  Then \(\scal_g = n(n-1) \kappa\) and \((V,g)\) is isometric to the warped product 
  \[(M \times [t_-, t_+], \varphi^2 g_M + \D{x} \otimes \D{x}),\] 
  where \(M = x^{-1}(t_0)\) for an arbitrary fixed \(t_0 \in [t_-,t_+]\) and \(g_M \coloneqq \varphi(t_0)^{-2} g|_{M}\) on \(M\).
  Furthermore, for any \(u \in \ker(\Callias_{\psi, s})\) and \(t \in [t_-,t_+]\), the restriction \(u|_{M \times \{t\}}\) is parallel with respect to the connection 
  \[ \widetilde{\nabla}_\xi \coloneqq \nabla_\xi + \tfrac{1}{2}\clm(\nabla_\xi \D x) \clm(\D  x), \qquad \xi \in \T (M \times \{t\}).\]
\end{thm}

The proof of this theorem is based on the following lemma, where we include the slightly more general situation of \(x\) being a not necessarily smooth Lipschitz function.
This will be of importance later for the rigidity of bands.
\begin{lem}\label{warped-product-rigidity-lemma}
  Let \((V,g)\) be an \(n\)-dimensional Riemannian spin band and \(E \to V\) be a Hermitian vector bundle endowed with a metric connection.
  Let \(x \colon V \to I\) be a \(1\)-Lipschitz function, and \(\varphi, h, f \colon I \to \R\), \(\kappa, \nu \colon V \to \R\) be smooth functions satisfying \cref{setup:WarpingFunctions}.
  Suppose furthermore that the conditions \labelcref{item:curv-bounds,item:width-function,item:mean-bound,item:callias-kernel} from the statement of \cref{thm:general-warped-product-rigidity} are satisfied.
  Then, for any \(u \in \ker(\Callias_{\psi, s})\), the section \(w \coloneqq \varphi(x)^{-\tfrac{1}{2}} u\) lies in \(\SobolevH^{1}_{\RelDiracInv, s}(V, S)\) and satisfies almost everywhere
  \begin{align} 
    \clm(\D x) w &= \RelDiracInv w, \label{eq:RelDiracInv-vs-dx} \\
    \nabla_\xi w &= \frac{f(x)}{n} \clm(\xi^\flat \wedge \D x) w, \label{eq:modified-section-cov-derivative}
  \end{align}
  for every vector field \(\xi\) on \(V\), where \(\clm(\xi^\flat \wedge \D x ) = \tfrac{1}{2}(\clm(\xi^\flat) \clm(\D x)  - \clm(\D x) \clm(\xi^\flat))\).
  In particular, \(|w|^2\) is a constant function and thus \(|u|^2\) is a constant multiple of the function \(\varphi(x)\).
  Moreover, \(|\D x| =1 \) almost everywhere and \(\scal_g = n(n-1) \kappa\).
\end{lem}
\begin{proof}
  We first observe that \(|\D x| \leq 1\) almost everywhere because \(x\) is \(1\)-Lipschitz.
  Let \(u \in \ker(\Callias_{\psi,s})\).
  By \labelcref{item:callias-kernel}, we can assume that \(u \neq 0\).
  Then, using \labelcref{eq:callias-estimate-boundary-condition} from \cref{thm:callias-estimate-with-boundary}, we obtain
  \begin{align}
  0 &= \int_V | \Callias_\psi u |^2\ \vol_V \nonumber \\
  &\geq \tfrac{n}{4(n-1)}\int_V (\scal_g - n(n-1) \kappa)|u|^2\vol_V +  \tfrac{n}{(n-1)}\int_V (\langle u, \Curv^E u \rangle - \nu |u|^2 )\vol_V \nonumber\\
  &\quad +  \int_V \left(\tfrac{n^2 \kappa}{4} + \tfrac{n \nu}{n-1} + f(x)^2 - f'(x)\right)|u|^2\ \vol_V + \int_V f'(x) \langle u, (1 + \clm(\D x) \RelDiracInv) u\rangle\ \vol_V \nonumber \\
  &\quad + \int_{\partial_- V} (\tfrac{n}{2} \mean_g - f(t_-))|u|^2\ \vol_{\partial_-V} + \int_{\partial_+ V} (\tfrac{n}{2} \mean_g + f(t_+))|u|^2\ \vol_{\partial_+V}. \label{eq:general-band-comparison-estimate}
  \end{align}
  Note that in the latter six integrals each integrand is non-negative---in the first two cases due to \labelcref{item:curv-bounds}, in the third this follows from \labelcref{eq:warping-balance}, in the fourth from~\(|\D x| \leq 1\), and for the two boundary integrals this is a consequence of the assumption \labelcref{item:mean-bound}.
  Thus all integrands appearing in \labelcref{eq:general-band-comparison-estimate} vanish (almost everywhere) and we are in the equality situation of \cref{thm:callias-estimate-with-boundary}.
  Since \(f' > 0\) and \(1 + \clm(\D x) \RelDiracInv \geq 0\) almost everywhere, we furthermore deduce from vanishing of the fourth integrand that \(\clm(\D x) \RelDiracInv u = - u\) almost everywhere and thus
  \begin{equation}
    \clm(\D x) u = \RelDiracInv u. \label{eq:epsilon-vs-dx}
  \end{equation}
  This already proves \labelcref{eq:RelDiracInv-vs-dx} for \(w \coloneqq \varphi(x)^{-\tfrac{1}{2}} u \).
  The identity \labelcref{eq:modified-section-cov-derivative} and the fact that \(|u|^2 = c \varphi\) for \(c = |w|^2\) constant is now a consequence of \cref{rem:modified_section} because \labelcref{eq:epsilon-vs-dx} holds almost everywhere.
  Moreover, since \(|u|^2 = c \varphi > 0\), vanishing of the first integrand in \labelcref{eq:general-band-comparison-estimate} implies that \(\scal_g = n(n-1) \kappa\).
  It also follows from \labelcref{eq:RelDiracInv-vs-dx} that \(|\D x|^2 |w|^2 = |\clm(\D x) w |^2 = | \RelDiracInv w|^2 = |w|^2\) and thus \(|\D x| =1\) almost everywhere.
\end{proof}

\begin{proof}[Proof of \cref{thm:general-warped-product-rigidity}]
  We will refer to \cref{warped-product-rigidity-lemma} and its proof in the following argument.
  First of all, we obtain that \(|\D x| = 1\).
  Since \(x\) is assumed to be smooth, this already implies that \(V\) is diffeomorphic to \(M \times [t_-, t_+]\) with \(g\) corresponding to \(g_x + \D x \otimes \D x\) for some family of Riemannian metrics \((g_x)_{x \in [t_-, t_+]}\) on \(M\).
  To prove the theorem, it remains to compute the Hessian of \(x\).
  To this end, we let \(0 \neq u \in \ker(\Callias_{\psi,\sigma})\) and let \(w \coloneqq \varphi(x)^{-\tfrac{1}{2}} u \) as in \cref{warped-product-rigidity-lemma}.
  Since \(\psi = f(x)\) is smooth, boundary elliptic regularity (see e.g.~\cite[Theorem~4.4]{Baer-Ballmann:Guide-Boundary-value-problems-Dirac}) implies that \(u\) and thus \(w\) are smooth sections.
  Using \labelcref{eq:RelDiracInv-vs-dx,eq:modified-section-cov-derivative} and the fact that \(\RelDiracInv\) is parallel, we compute for any smooth vector field \(\xi\) on \(V\),
  \begin{align*}
    \clm(\nabla_\xi \D x) w &= \underbrace{\nabla_\xi(\clm(\D x) w)}_{=\RelDiracInv \nabla_\xi w} - \clm(\D x) \nabla_\xi w \\
    &\underset{\text{\labelcref{eq:modified-section-cov-derivative}}}{=}  \frac{f(x)}{n} \RelDiracInv \clm(\xi^\flat \wedge \D x) w - \frac{f(x)}{n} \clm(\D x) \clm(\xi^\flat \wedge \D x) w \\
    &=  \frac{f(x)}{n} \left(\clm(\xi^\flat \wedge \D x) \RelDiracInv w + \clm(\xi^\flat \wedge \D x) \clm(\D x) w \right) \\
    &= \frac{2 f(x)}{n} \clm(\xi^\flat \wedge \D x) \clm(\D x) w \\
    &= -\frac{2 f(x)}{n} \clm(\xi^\flat - \D x(\xi)\ \D x) w
    = h(x) \clm(\xi^\flat - \D x(\xi)\ \D x) w.
  \end{align*}
  Since \(w\) vanishes nowhere, this implies \(\nabla_\xi \D x = h(x)(\xi^\flat - \D x (\xi) \D{x})\) for each vector field \(\xi\). 
  Hence the Hessian of \(x\) is given by
  \begin{equation}
    \nabla^2 x  = h(x) (g - \D x \otimes \D x).\label{eq:warped-product-hessian}
  \end{equation}

  Using standard formulas for Riemannian distance functions (in the sense of~\cite[Section~3.2.2]{Petersen}), the identity \labelcref{eq:warped-product-hessian} implies for the Lie derivative of \(g_x\) that 
  \[\LieDeriv_{\partial_x} g_x = 2 h(x) g_x,\]
  see for instance~\cite[Proposition~3.2.11 (1)]{Petersen}.
  Let \(\tilde{g}_x = \varphi(x)^2 g_M\), where \(g_M \coloneqq \varphi(t_0)^{-2}g|_M\).
  Then \(\tilde{g}_x\) satisfies the same differential equation because
  \begin{equation*}
    \LieDeriv_{\partial_x} \tilde{g}_x = 2 \varphi'(x) \varphi(x) g_{M} = 2 \tfrac{\varphi'(x)}{\varphi(x)} \tilde{g}_x = 2 h(x) \tilde{g}_x. 
  \end{equation*}
  Since \(\tilde{g}_{t_0} = g_{t_0} = g_{M}\) this implies that \(g_x = \tilde{g}_x = \varphi(x)^2 g_{M}\), as desired.
  
  To see the final statement, observe that by \cref{rem:equality-kernel} we have
  \[
     \quad \nabla_\xi u = \frac{f(x)}{n} \clm(\xi^\flat) \RelDiracInv u = \frac{f(x)}{n} \clm(\xi^\flat) \clm(\D x) u.
  \]
  Since \(\nabla_\xi \D x = h(x) \xi = - \frac{2}{n} f(x) \xi\) for any vertical tangent vector \(\xi\), this implies that \(u|_{M \times \{t\}}\) is parallel with respect to \(\widetilde{\nabla}\).
\end{proof}

\section{Rigidity of bands}\label{sec:band-rigidity}
In this section, we use the general results from the previous section to deduce a rigidity result for \(\Ahat\)-bands subject to scalar and mean curvature bounds.
\begin{thm}\label{thm:band-rigidity}
  Let \((V,g)\) be an \(n\)-dimensional band which is a spin manifold and satisfies \(\Ahat(\partial_- V) \neq 0\).
  Suppose furthermore that there exist \(-\pi/n < t_- < t_+ < \pi/n\) such that
  \begin{myenumi}
    \item \(\scal_g \geq n(n-1)\); \label{item:scal-bound-band}
    \item \(\width(V,g) \geq t_+ - t_-\);
    \item \(\mean_g|_{\partial_\pm V} \geq \mp \tan({n t_\pm}/2)\), where \(\mean_g\) denotes the mean curvature of \(\partial V\). \label{item:mean-bound-band}
  \end{myenumi}
  Then \((V,g)\) is isometric to a warped product \((M \times [t_-, t_+], \varphi^2 g_M + \D{x} \otimes \D{x})\), where \(\varphi(t) = \cos\left(n t/{2}\right)^{2/n}\) and \(g_M\) is some Riemannian metric on \(M\) which carries a non-trivial parallel spinor.
  In particular, \(g_M\) is Ricci-flat.
\end{thm}
\begin{proof}
 We first prepare a particular case of \cref{setup:WarpingFunctions} we wish to apply.
 To this end, let
\[
  \varphi \colon \left(-\tfrac{\pi}{n}, \tfrac{\pi}{n}\right) \to (0, \infty), \quad \varphi(t) \coloneqq \cos\left(\tfrac{n t}{2}\right)^{\frac{2}{n}}.
\]
Then
\[
  h(t) \coloneqq \frac{\varphi'(t)}{\varphi(t)} = -\tan\left(\tfrac{n t}{2}\right), \quad h'(t) = - \frac{n}{2} \frac{1}{\cos\left(\tfrac{nt}{2}\right)^2} < 0,
  \]
and
\begin{equation}
  \frac{n^2}{4} + f(t)^2 - f'(t) = 0, \label{eq:band-warping-equality}
\end{equation}
where \(f = -\frac{n}{2} h\).
Thus \eqref{eq:warping-balance} is satisfied with \(\kappa \equiv 1\) and \(\nu \equiv 0\) and we are in \cref{setup:WarpingFunctions}.
Next we use \cref{lem:width-function} to choose a width function \(x \colon V \to [t_-, t_+]\) which is smooth near the boundary of \(V\).
We let \(\psi = f \circ x\) and form the Callias-type operator \(\Callias_{\psi}  = \Dirac + \psi \RelDiracInv\).

We choose signs \(s \colon \partial M \to \{\pm 1\}\) as in \cref{ex:band-index-conputation}.
Then, since \(\Ahat(\partial_{-} V) \neq 0\), we deduce from \cref{ex:band-index-conputation} that \(\ind(\Callias_{\psi, s}) \neq 0\).
In particular, \(\ker(\Callias_{\psi, s}) \neq 0\).

At this point, if \(x\) was smooth everywhere, the result would follow readily from \cref{thm:general-warped-product-rigidity} because \labelcref{item:mean-bound-band} says that \(\mean_g|_{\partial_\pm V} \geq \pm h(t_\pm)\).
However, since we do not know this a priori, we need to supply an argument ensuring that \(x\) is indeed smooth everywhere.
To this end, we fix \(0 \neq u \in \ker(\Callias_{\psi,s})\) and apply \cref{warped-product-rigidity-lemma}.
From \labelcref{eq:square_of_callias,eq:RelDiracInv-vs-dx,eq:band-warping-equality}, we thus deduce
\begin{align*}
0 = \Callias_\psi^2 u &= \Dirac^2 u + f'(x) \clm(\D x) \RelDiracInv u + f(x)^2 u  \\
&= \Dirac^2 u -f'(x) u + f(x)^2 u \\
&= \Dirac^2 u - \tfrac{n^2}{4} u.
\end{align*}
Then interior elliptic regularity for the operator \(\Dirac^2 - n^2/4\) implies that \(u\) is smooth in the interior of \(V\).
Since, by \cref{warped-product-rigidity-lemma}, \(u\) is nowhere vanishing and we have the equality \(\clm(\D x) u = \RelDiracInv u\), this implies that the covector field \(\D x\) must also be smooth in the interior of \(V\).
Since we already know that \(x\) is smooth near the boundary, this just means that the function \(x\) is smooth everywhere and so we can indeed apply \cref{thm:general-warped-product-rigidity}. 

We conclude that \((V,g)\) is isometric to a warped product \((M \times [t_-, t_+], \varphi^2 g_M + \D{x} \otimes \D{x})\) and \(\scal_g \equiv n(n-1)\).
Moreover, the final statement of \cref{thm:general-warped-product-rigidity} implies that any \(0 \neq u \in \ker(\Callias_{\psi, s})\) restricts to a nowhere-vanishing parallel spinor on each fiber (compare~\eqref{eq:boundary-connection}).
It is a well-known fact that the existence of a parallel spinor forces the Ricci curvature to vanish, see for instance~\cite[Corollary~2.8]{SpinorialApproach}.
\end{proof}

\begin{cor}\label{cor:band-rigidity-concrete}
Let \((V,g)\) be an \(n\)-dimensional band which is a spin manifold and satisfies \(\Ahat(\partial_- V) \neq 0\).
Suppose that \(\scal_g \geq n(n-1)\).
  Let \(0 < d < \pi/n\) and assume furthermore that one of the following conditions holds:
  \begin{itemize}
    \item either \(\width(V,g) \geq 2d\) and \(\mean_g|_{\partial V} \geq -\tan({n d}/2)\),
    \item or \(\width(V,g) \geq d\) and \(\mean_g|_{\partial_- V} \geq 0\), \(\mean_g|_{\partial_+ V} \geq -\tan({n d}/2)\).
  \end{itemize}
  Then \((V,g)\) is isometric to a warped product \((M \times I, \varphi^2 g_M + \D{x} \otimes \D{x})\), where either \(I = [-d,d]\) or \(I=[0,d]\), \(\varphi(t) = \cos\left(n t/{2}\right)^{2/n}\) and \(g_M\) is some Riemannian metric on \(M\) which carries a non-trivial parallel spinor.
  In particular, \(g_M\) is Ricci-flat.
\end{cor}
\begin{proof}
  This follows immediately from \cref{thm:band-rigidity} by setting \(t_\pm = \pm d\) in the first case, and \(t_- = 0\) and \(t_+ = d\) in the second.
\end{proof}

\section{Scalar-mean extremality and rigidity of warped products}
In this section, we prove our general extremality and rigidity results for logarithmically concave warped products.

As a preparation for the proof of the main theorem, we discuss a particularly relevant example of a twisting bundle \(E \to V\) to be used in \cref{thm:general-warped-product-rigidity}, namely the fiberwise spinor bundle on a warped product.

\begin{rem}[Twisting with the fiberwise spinor bundle] \label{fiberwise-spinor}
  Consider a warped product band
  \[
    (V \coloneqq M \times [t_-, t_+], g \coloneqq \varphi^2 g_M + \D{x} \otimes \D{x})
  \]
   and let \(\ReducedSpinBdl \to V\) be the spinor bundle with respect to the metric \(g\).
   Then we let \(E_0 \coloneqq \ReducedSpinBdl\) be the same bundle endowed with the same bundle metric but with the \enquote{fiberwise spinor connection}
   \[
    \nabla^{E_0}_X \coloneqq \nabla^{\ReducedSpinBdl}_X + \frac{1}{2} \clm(\nabla_X \D x) \clm(\D x) \quad (X \in \T V).
   \]
   With this connection, each restriction \(E_0|_{M \times \{t\}}\) is precisely the spinor bundle of \(M \times \{t\}\) (if \(n\) is odd) or two copies of it (if \(n\) is even), compare also \eqref{eq:boundary-connection}.
   Moreover, a direct calculation using the warped product structure shows that the curvature tensor of \(\nabla^{E_0}\) satisfies \(\FullCurv^{E_0}_{\partial_x, X} = 0\) for any tangent vector \(X \in \T V\).
   In other words, the connection \(\nabla^{E_0}\) is chosen in such a way that only the vertical directions contribute to its curvature.
   Consequently, if we form the twisted spinor bundle \(\ReducedSpinBdl \otimes E_0\), the corresponding curvature endomorphism from the Bochner--Lichnerowicz--Weitzenböck formula (see~\labelcref{Weitzenbock,E:Band-Lichnerowicz}) satisfies 
   \[
    \Curv^{E_0}|_{M \times \{t\}} = \sum_{i=1}^{n-1} \clm(e^i)\clm(e^j) \otimes \FullCurv^{\nabla^{E_0}}_{e_i,e_j} = \sum_{i=1}^{n-1} \clm(e^i)\clm(\D x)\clm(e^j)\clm(\D x) \otimes \FullCurv^{\nabla^{E_0}}_{e_i,e_j} = \Curv^{E_0|_{M \times \{t\}}},
   \]
   where \(e_1, \dotsc, e_{n-1}\) is a local orthonormal frame of \(\T M\) and \(\Curv^{E_0|_{M \times \{t\}}}\) denotes the Weitzenböck curvature endomorphism on the fiber \(M \times \{t\}\) of the twisting bundle \(E_0|_{M \times \{t\}}\) (or two copies thereof).
   Finally, since the metric on \(M \times \{t\}\) is simply the constant multiple \(\varphi(t)^2 g_M\) of the metric \(g_M\), we can identify each restriction \(E_0|_{M \times \{t\}}\) with the spinor bundle on \(M\) (or two copies of it) and with respect to this identification, we obtain
   \begin{equation}
    \Curv^{E_0}|_{M \times \{t\}} = \varphi(t)^{-2}\Curv^{M, E_0}, \label{weitzenboeck-curv-fiberwise-spinor}
   \end{equation}
   where \(\Curv^{M, E_0}\) denotes the curvature endomorphism on \((M, g_M)\) associated to using the spinor bundle on \(M\) itself as a twisting bundle (or two copies of each).
\end{rem}

We now state and proof the main result of this section.

\begin{thm} \label{thm:goette-semmelmann-band-rigidity}
  Let \(n\) be odd and \((N, g_N)\) be an \((n-1)\)-dimensional Riemannian spin manifold of non-vanishing Euler-characteristic whose Riemannian curvature operator is non-negative.
  Moreover, let \(\varphi \colon [t_-, t_+] \to (0,\infty)\) be a strictly logarithmically concave function and consider the warped product metric \(g_0 = \varphi^2 g_N + \D y \otimes \D y\) on \(V_0 \coloneqq N \times [t_-, t_+]\).
  Let \((V,g)\) be an \(n\)-dimensional Riemannian spin band and \(\Phi \colon (V, g) \to (V_0, g_0)\) a smooth band map such that
  \begin{myenumi}
    \item \(\Phi\) is \(1\)-Lipschitz and of non-zero degree;
    \item \(\scal_g \geq \scal_{g_0} \circ\ \Phi\);
    \item \(\mean_g|_{\partial_\pm V} \geq \mean_{g_0}|_{\partial_\pm V_0} = \pm h(t_\pm)\), where \(h = \varphi'/\varphi\).
  \end{myenumi}
  Then \(\scal_g = \scal_{g_0} \circ\ \Phi\) and \((V,g)\) is isometric to a warped product 
  \[(M \times [t_-, t_+], \varphi^2 g_M + \D{x} \otimes \D{x}),\] 
  where \(x = y \circ \Phi\), \(M = x^{-1}(t_0)\) for any \(t_0 \in [t_-,t_+]\) and \(g_M \coloneqq \varphi(t_0)^{-2} g|_{M}\) on \(M\).
  Moreover, we have \(\scal_{g_M} = \scal_{g_N} \circ\ \Phi|_{M}\) in this case.
  
  If, furthermore, the metric on \(N\) satisfies \(\Ric_{g_N} > 0\), then \(\Phi\) is an isometry under the above hypotheses.
\end{thm}
\begin{proof}
  The main idea of the proof is to apply the argument of \citeauthor{Goette-Semmelmann}~\cite{Goette-Semmelmann} fiberwise in combination with~\cref{thm:general-warped-product-rigidity}.
  Since \(\Phi\) is a \(1\)-Lipschitz band map, the function \(x = y \circ \Phi \colon V \to [t_-, t_+]\) is a width function.
  We need to verify that we are in an instance of \cref{setup:WarpingFunctions}.
To this end, we set
\begin{equation}
  \kappa \coloneqq \frac{\scal_{g_0} \circ \Phi}{n(n-1)}, \qquad
\nu \coloneqq - \frac{\scal_{g_N} \circ \proj_{N} \circ \Phi}{4 \varphi(x)^2}, \label{eq:choice-warping-functions-general}
\end{equation}
where \(\proj_N \colon V_0 \to N\) is the projection onto the first factor.
It is now a consequence of the discussion in \cref{rem:warping-balance-motivation} that
\begin{equation}
  \frac{n^2 \kappa}{4} + \frac{n \nu}{n-1} + f(x)^2 - f'(x) = 0, \label{eq:warping-balance-general-warped}
\end{equation}
where \(f \coloneqq - \tfrac{n}{2} h\).
In particular, this choice of functions satisfies \cref{setup:WarpingFunctions}.
  
Now we consider the fiberwise spinor bundle \(E_0 \to V_0\) constructed as in \cref{fiberwise-spinor} and let \(E \coloneqq \Phi^\ast E\) the pull-back bundle on \(V\).
Now the main estimate of \citeauthor{Goette-Semmelmann}~\cite[Section~1.1]{Goette-Semmelmann} together with the description of the Weitzenböck curvature endomorphism of \(E_0\) from \labelcref{weitzenboeck-curv-fiberwise-spinor} shows that we precisely have the estimate
  \[
    \Curv^{E} \geq \nu.
  \]
   For each \(t \in [t_-,_+]\), the twisted Dirac operator \(\ReducedSpinDirac_{N \times \{t\}, E_0|_{N \times \{t\}}}\) is the Euler characteristic operator of \(N\) and thus has non-trivial index because \(n-1\) is even.
   Since the degree of \(\Phi\) is non-zero, it follows that the index of \(\ReducedSpinDirac_{\partial_- V, E|_{\partial_- V}}\) is also non-zero.
   Hence, \cref{ex:band-index-conputation} shows that the Callias operator \(\Callias_{\psi, s}\) considered in \cref{thm:general-warped-product-rigidity} has non-trivial index and hence non-trivial kernel.
   Thus \cref{thm:general-warped-product-rigidity} applies and we obtain \(\scal_g = \kappa n(n-1)\) and \((V,g)\) is isometric to a warped product 
   \[ (M \times [t_-,t_+], \varphi^2 g_{M} + \D x \otimes \D x),\]
    where \(M = x^{-1}(t_0)\) for some arbitrary but fixed \(t_0 \in [t_-, t_+]\) and some Riemannian metric \(g_M\) on \(M\).
   Using the warped product structure, we obtain
   \[
   \kappa n(n-1) = \scal_{g} = \frac{\scal_{g_M}}{\varphi(x)^2} - 2(n-1)h'(x) - n(n-1) h(x)^2.
   \]
  Together with \labelcref{eq:warping-balance-general-warped} this completely determines the scalar curvature of \(g_M\) and we obtain \(\scal_{g_M}= \scal_{g_N} \circ \Phi|_{M}\).
  This proves the first part of the theorem.
  
  To prove the second part, we first observe that due to \(x = y \circ \Phi\), under the isometry \(V \cong M \times [t_-, t_+]\), the map \(\Phi\) is of the form \((p, x) \mapsto (\Phi_{x}(p), x)\), where \(\Phi_t \coloneqq \Phi|_{M \times \{t\}} \colon M  \to N \).
  Since \(\| \T \Phi\| \leq 1\) and \(\T \Phi(0, \partial_x) = (\tfrac{\partial}{\partial x}\Phi_x, \partial_x) \), it follows that \(\tfrac{\partial}{\partial x}\Phi_x = 0\), that is, \(\Phi_t = \Phi_{t_0}\) for all \(t\).
  Moreover, \(\Phi_{t_0} \colon M \to N\) is \(1\)-Lipschitz with respect to the metric \(\varphi(t_0)^2 g_M\) and \(\varphi(t_0)^2 g_{N}\).
  Thus the same holds with respect to the metrics \(g_M\) and \(g_N\).
  If we assume that \(\Ric_{g_N} > 0\), then the rigidity argument of~\citeauthor{Goette-Semmelmann}, see \cite[Section~1.2]{Goette-Semmelmann} (and note that our \(\Phi_{t_0}\) is length-non-increasing), implies that \(\Phi_{t_0}\) is an isometry.
   Together with the warped product structure all of this implies that \(\Phi\) itself is an isometry.
\end{proof}

Restricting to the special case where \(\Phi\) is the identify map immediately yields the following corollary.
The notions of \emph{scalar-mean extremality} and \emph{-rigidity} are defined in the introduction in \cref{subsec:intro-rigidity}.

\begin{cor} \label{goette-semmelmann-band-extremality}
  Let \(n\) be odd and \((M, g_M)\) be an \((n-1)\)-dimensional Riemannian spin manifold of non-vanishing Euler-characteristic whose Riemannian curvature operator is non-negative.
  Let \(\varphi \colon [t_-, t_+] \to (0,\infty)\) be a smooth strictly logarithmically concave function and consider the warped product metric \(g_V = \varphi^2 g_M + \D x \otimes \D x\) on \(V \coloneqq M \times [t_-, t_+]\).
  Then any metric \(g\) on \(V\) which satisfies 
  \begin{myenumi}
    \item \(g \geq g_V\),
    \item \(\scal_g \geq \scal_{g_V}\),
    \item \(\mean_g \geq \mean_{g_V}\)
  \end{myenumi}
  is itself a warped product \(g = \varphi^2 \tilde{g}_M + \D x \otimes \D x\) for some metric \(\tilde{g}_M\) on \(M\) which satisfies \(\scal_{\tilde{g}_M} = \scal_{g_M}\).
  In particular, \(g_V\) is scalar-mean extremal.
  
  If, in addition, the metric \(g_M\) satisfies \(\Ric_{g_M} > 0\), then \(g_V\) is scalar-mean rigid.
\end{cor}

In particular, the main theorem and corollary of this section are fully applicable to strictly log-concave warped products over even-dimensional spheres.
This corresponds to a fiberwise application of Llarull's result~\cite{LLarull}.
In the following, we single out one important special class of examples, namely annuli in simply-connected space forms.

Indeed, let \(\kappa \in \R\) be fixed and \((M_\kappa, g_\kappa)\) be the \(n\)-dimensional simply connected space form of constant sectional curvature \(\kappa\).
To apply the theorem, we recall the description of \((M_\kappa, g_\kappa)\) as a warped product over the sphere.
Choose a base-point \(p_0 \in M_\kappa\).
Let \(\sn_\kappa\) be the unique solution to the initial value problem \(\varphi'' + \kappa \varphi =0\), \(\varphi(0) = 0\), \(\varphi'(0) = 1\).
Similarly, \(\cs_\kappa\) denotes the unique solution to the same differential equation but with initial values \(\varphi(0) = 1\), \(\varphi'(0) = 0\).
If \(\kappa > 0\), we let \(p_\infty \in M_\kappa\) be the point opposite to \(p_0\) and set \(M_\kappa^\prime \coloneqq M_\kappa \setminus \{p_0, p_\infty\}\).
  If \(\kappa \leq 0\), we let \(M_\kappa^\prime = M_\kappa \setminus \{p_0\}\).
On \(M_\kappa^\prime \cong \Sphere^{n-1} \times (0, t_\infty)\) the metric \(g_\kappa\) appears as the warped product 
\[g_\kappa = \sn_\kappa^2 g_{\Sphere^{n-1}} + \D x \otimes \D x,\]
where \(t_\infty\) is chosen such that \(I \coloneqq (0, t_\infty)\) is a maximal interval on which \(\sn_\kappa\) remains positive.
This means \(t_\infty = +\infty\) for \(\kappa \leq 0\) and \(t_\infty = \pi / \sqrt{\kappa}\) for \(\kappa > 0\).
Moreover, \(\log(\sn_\kappa)'' = -\tfrac{1}{\sn_\kappa^2} < 0\), that is, \(\sn_\kappa\) is strictly logarithmically concave.
This means that \cref{thm:goette-semmelmann-band-rigidity} is applicable to the metric \(g_\kappa\).
Given \(0 < t_- < t_+ < t_\infty\), we consider the annulus 
\[\Annulus_{t_-, t_+} \coloneqq \{ p \in M_\kappa \mid t_- \leq d_{g_\kappa}(p,p_0) \leq t_+ \} \subset M_\kappa.\] 
We will view \(\Annulus_{t_-, t_+} \) as a band with \(\partial_{\pm} \Annulus_{t_-, t_+} = \Sphere_{t_\pm} \coloneqq \{p \in M_\kappa \mid d_{g_\kappa}(p, p_0) = t_\pm\}\).
Furthermore, we set \(\ct_\kappa = \cs_\kappa / \sn_\kappa\).
Then the mean curvature of \(\partial_{\pm} \Annulus_{t_-, t_+}\) is equal to  \(\pm \ct_\kappa(t_\pm)\).
We thus deduce the following consequences of \cref{thm:goette-semmelmann-band-rigidity}.

\begin{cor} \label{annulus-rigidity}
  Let \(n \geq 3\) be odd and \((M_\kappa, g_\kappa)\) the \(n\)-dimensional simply connected space form of constant sectional curvature \(\kappa \in \R\).
  Let \(0 < t_- < t_+ < t_\infty\) and consider an annulus \(\Annulus_{t_-, t_+}\) as above.
  Let \((V,g)\) be an \(n\)-dimensional spin band and \(\Phi \colon V \to \Annulus_{t_-, t_+}\) be a smooth band map such that
  \begin{myenumi}
    \item \(\Phi\) is \(1\)-Lipschitz and of non-zero degree,
    \item \(\scal_g \geq \scal_{g_\kappa} = \kappa n(n-1)\),
    \item \(\mean_g|_{\partial_\pm V} \geq \mean_{g_\kappa}|_{\partial_\pm \Annulus_{t_-, t_+}} = \pm \ct_\kappa(t_\pm)\).
  \end{myenumi}
  Then \(\Phi\) is an isometry.
\end{cor}

\begin{cor} \label{annulus-extremal-rigidity}
  Let \(n \geq 3\) be odd and \((M_\kappa, g_\kappa)\) the \(n\)-dimensional simply connected space form of constant sectional curvature \(\kappa \in \R\).
  Let \(0 < t_- < t_+ < t_\infty\) and consider the annulus 
  \[\Annulus_{t_-, t_+} \coloneqq \{ p \in M_\kappa \mid t_- \leq d_{g_\kappa}(p,p_0) \leq t_+ \}\] 
  around some base-point \(p_0 \in M_\kappa\).
  Then any Riemannian metric \(g\) on \(\Annulus_{t_-, t_+}\) which satisfies
  \begin{myenumi}
    \item \(g \geq g_\kappa\),
    \item \(\scal_g \geq \scal_{g_\kappa} = \kappa n(n-1)\),
    \item \(\mean_g|_{\Sphere_{t_\pm}} \geq \mean_{g_\kappa}|_{\Sphere_{t_\pm}} = \pm \ct_\kappa(t_\pm)\)
  \end{myenumi}
  is equal to \(g_\kappa\).
  That is, \(g_\kappa\) is scalar-mean rigid on \(\Annulus_{t_-, t_+}\).
\end{cor}

\printbibliography
\end{document}